\documentclass[11pt]{amsart}
\usepackage{amsaddr}
\usepackage{graphicx,epsfig}

\usepackage{mathtools}
\usepackage{enumerate}

\usepackage{color}
\usepackage{verbatim}

 \usepackage[bookmarks=true,
bookmarksnumbered=true, breaklinks=true,
pdfstartview=FitH, hyperfigures=false,
plainpages=false, naturalnames=true,
colorlinks=false,
pdfpagelabels]{hyperref}

\makeatletter
\newtheorem*{rep@theorem}{\rep@title}
\newcommand{\newreptheorem}[2]{%
\newenvironment{rep#1}[1]{%
 \def\rep@title{#2 \ref{##1}}%
 \begin{rep@theorem}}%
 {\end{rep@theorem}}}
\makeatother
\newreptheorem{theorem}{Theorem}
\newreptheorem{corollary}{Corollary}

\newcommand{\R}{\mathbb{R}} 

\newcommand{\Rn}{\mathbb{R}^n} 

\newcommand \B[1][n]{B_2^{#1}}
\newcommand{\N}{\mathbb{N}}
\newcommand{\vol}[1][n]{\operatorname{vol}_{#1}}
\newcommand{\Vol}{\operatorname{vol}_{nm}}
\def \s{\mathbb{S}^{n-1}}
\def \S{\mathbb{S}^{nm-1}}

\newcommand {\conbod}[1][n,m] {\mathcal{K}^{#1}}
\newcommand {\conbodo}[1][n,m] {\mathcal{K}^{#1}_o}
\newcommand {\conbodio}[1][n,m] {\mathcal{K}^{#1}_{(o)}}
\newcommand {\sta}[1][n,m] {\mathcal{S}^{#1}}

\renewcommand{\P}[1][Q]{\Pi_{#1,p}}
\newcommand{\Pinf}[1][Q]{\Pi_{#1,\infty}}
\newcommand{\PP}[1][Q]{\Pi_{#1,p}^{\circ}}

\newcommand{\pp}[1][\alpha_1,\alpha_2]{\Pi_{#1,p}^\circ}
\newcommand{\G}[1][Q]{\Gamma_{#1,p}}

\newcommand {\M}[1][n,m] {\operatorname{M}_{#1}(\R)}

\newcommand {\LYZ}[1]{\P{#1}}
\newcommand {\LYZP}[1]{\PP #1}

\mathtoolsset{showonlyrefs}

\newtheorem{theorem}{Theorem}[section]
\newtheorem{definition}[theorem]{Definition}
\newtheorem{lemma}[theorem]{Lemma}
\newtheorem{proposition}[theorem]{Proposition}
\newtheorem{corollary}[theorem]{Corollary}

\title[Higher Order Inequalities]{Higher-Order $L^p$ Isoperimetric  and Sobolev Inequalities}

\author[Haddad]{Juli\'an Haddad}
\address{Departamento de An\'alisis Matem\'atico \\ Universidad de Sevilla \\ 
Sevilla, Spain 29208}
\email{jhaddad@us.es}

\author[Langharst]{Dylan Langharst}
\address{Institut de Math\'ematiques de Jussieu \\ Sorbonne Universit\'e \\
Paris, France 75252}
\email{dylan.langharst@imj-prg.fr}

\author[Putterman]{Eli Putterman}
\address{School of Mathematical Sciences \\ 
Tel Aviv University\\ Tel Aviv 66978, Israel}
\email{putterman@mail.tau.ac.il}

\author[Roysdon]{Michael Roysdon}
\address{Department of Mathematics, Applied Mathematics, and Statistics\\ Case Western Reserve University \\ Cleveland, OH 44106, USA}
\email{mar327@case.edu}

\author[Ye]{Deping Ye}
\address{Department of Mathematics and Statistics \\ Memorial University of Newfoundland \\
St. John’s, Newfoundland, A1C 5S7, Canada}
\email{deping.ye@mun.ca}


\begin{document}

\begin{abstract} Schneider introduced an inter-dimensional difference body operator on convex bodies and proved an associated inequality. In the prequel to this work, we showed that this concept can be extended to a rich class of operators from convex geometry and proved the associated isoperimetric inequalities. The role of cosine-like operators, which generate convex bodies in $\R^n$ from those in $\R^n$, were replaced by inter-dimensional simplicial operators, which generate convex bodies in $\R^{nm}$ from those in $\R^{n}$ (or vice versa). In this work, we treat the $L^p$ extensions of these operators, and, furthermore, extend the role of the simplex to arbitrary $m$-dimensional convex bodies containing the origin. We establish $m$th-order $L^p$ isoperimetric inequalities, including the $m$th-order versions of the $L^p$ Petty projection inequality, $L^p$ Busemann-Petty centroid inequality, $L^p$ Santal\'o inequalities, and $L^p$ affine Sobolev inequalities. As an application, we obtain isoperimetric inequalities for the volume of the operator norm of linear functionals $(\R^n, \|\cdot\|_E) \to (\R^m, \|\cdot\|_F)$.
\end{abstract}

\maketitle

\tableofcontents

\section{Introduction}

The isoperimetric inequality in $\Rn$ asserts that, for every bounded domain $D \subset \Rn$ with non-empty interior and piecewise smooth boundary, one has
\begin{equation} \label{eq:isoclassic}
\left(\frac{\vol[n-1](\partial D)}{\vol[n-1](\mathbb{S}^{n-1})}\right)^{\frac{1}{n-1}} \geq \left(\frac{\vol(D)}{\vol(\B)} \right)^{\frac{1}{n}}.
\end{equation}
Here $\B =\{x \in \Rn \colon |x| \leq 1\}$ is the usual Euclidean unit ball in $\Rn$, where $\Rn$ is the standard $n$-dimensional Euclidean space and $|x|$ is the Euclidean norm of $x\in \Rn$, $\partial D$ denotes the boundary of $D$,  $\mathbb{S}^{n-1}=\partial \B$, and $\vol[d](E)$ is the $d$-dimensional Hausdorff measure of the measurable set $E\subset\Rn$ (e.g., $\vol[n](E)$ and $\vol[n-1](\partial E)$ denote the volume and the surface area of $E$, respectively). It is well-known that the isoperimetric inequality \eqref{eq:isoclassic} is equivalent to the $L^1$ Sobolev inequality for compactly supported smooth functions (see \cite[Chapter~5]{EvansGar}). In this work, we will be concerned with other isoperimetric-type inequalities from convex geometry, and the Sobolev-type inequalities they imply.

A set $K$ in $\Rn$ is said to be a \textit{convex body} if it is a compact set with non-empty interior such that $(1-\lambda)y+\lambda x\in K$ for every $x,y\in K$ and $\lambda\in[0,1]$. The set of convex bodies in $\Rn$ will be denoted by both by $\conbod[n]$ and $\conbod[n,1]$ (the reason for the non-standard notation will be made clear shortly). A convex body $K$ is origin symmetric if $K=-K$, and is said to be symmetric if a translate of $K$ is origin symmetric. The set of convex bodies in $\Rn$ containing the origin will be denoted both by $\conbodo[n]$ and $\conbodo[n,1]$, and those containing the origin in their interiors will be denoted by $\conbodio[n]$ and $\conbodio[n,1]$.

We set $\omega_q = \frac{\pi^\frac{q}{2}}{\Gamma\left(1+\frac{q}{2}\right)}$ for $q>0$. When $q$ is an integer $n$, we have $\omega_n=\vol(\B)$. In \cite{CMP71}, Petty established the following affine isoperimetric inequality: given $K \in \conbod[n]$, one has, with $\langle \cdot , \cdot \rangle$ the Euclidean inner-product on $\Rn$, 
 \begin{equation} \label{eq:PettyClassical}
\vol(K)^{\frac{n-1}{n}} \!\left(\frac{1}{n} \int_{\s} \!\left(\frac{1}{2}\int_{\partial K} |\langle n_K(v),\theta \rangle| dv \right)^{-n}d\theta \right)^{\frac{1}{n}}\! \leq\! \frac{\omega_n}{\omega_{n-1}},
 \end{equation}
 with equality if, and only  if, $K$ is an ellipsoid. Here, $d\theta$ denotes the spherical Lebesgue measure, $dv=d\mathcal{H}^{n-1}(v)$ denotes the $(n-1)$-dimensional Hausdorff measure on $\partial K$, and $n_K(x)$ denotes the outer unit normal vector at $x \in \partial K$, which is well defined for almost every $x\in\partial K$.
 Inequality \eqref{eq:PettyClassical}, referred to as {\bf Petty's projection inequality}, is invariant under affine transforms of $\Rn$, and it implies the classical isoperimetric inequality \eqref{eq:isoclassic} when $K$ is a convex body. 

 Zhang \cite{GZ99} generalized Petty's projection inequality \eqref{eq:PettyClassical} to smooth compact domains in order to establish an affine version of the classical $L^1$ Sobolev inequality. More precisely, he proved the following: if $f \colon \Rn \to [0,\infty)$ is a $C^1$ compactly supported function, then 
 \begin{equation}\label{eq:ZhangAffineSob}
\frac{n\omega_n}{2\omega_{n-1}}\left(\int_{\s} \left(\int_{\Rn} |\langle \nabla f(v), \theta \rangle | dv\right)^{-n} \frac{d\theta}{n\omega_n} \right)^{-\frac{1}{n}} \geq n\omega_n^\frac{1}{n} \|f\|_{\frac{n}{n-1}},
 \end{equation}
 with equality when $f$ is the characteristic function of an ellipsoid. Here for $p \geq 1$, $\|f\|_p$ denotes the usual $L^p$ norm of the function $f$ and $\nabla f$ denotes the gradient of $f$; the gradient of a characteristic function is understood via approximation.  Zhang \cite{GZ99} showed that the inequality \eqref{eq:ZhangAffineSob} is equivalent to the generalized Petty's projection inequality he established, and directly implies the classical $L^1$ Sobolev inequality.  In \cite{TW12}, Wang generalized the inequality \eqref{eq:ZhangAffineSob} to the space of functions of bounded variation on $\Rn$.

 At the turn of the millennium, the seminal papers \cite{LYZ00,LYZ02} of  Lutwak, Yang, and Zhang continued the studies of Petty \cite{CMP71} and Zhang \cite{GZ99}, respectively, to the $L^p$ setting, $p >1$. In particular, in \cite{LYZ02}, a sharp affine version of the Aubin-Talenti $L^p$ Sobolev inequality \cite{Aubin1,Talenti1} was established. Denoting by $W^{1,p}(\Rn)$ the $L^p$ Sobolev space on $\Rn$ (see Section~\ref{sec:functionspaces} below), the sharp $L^p$ affine Sobolev inequality of Lutwak, Yang, and Zhang is as follows: given any function $f \in W^{1,p}(\R^n)$, with $1 < p < n$,  one has 
 \begin{equation} \label{eq:LYZAffineSobolev}
 \begin{split}
\left(\frac{n\omega_n\omega_{p-1}}{2\omega_{n+p-2}}\right)^\frac{1}{p}\!\left(\int_{\s}\! \left(\int_{\Rn}|\langle \nabla f(v),\theta\rangle|^p dv \right)^{-\frac{n}{p}}\!\frac{d\theta}{n\omega_n}\right)^{-\frac{1}{n}}\! \geq \!a_{n,p}\|f\|_{\frac{pn}{n-p}}.
\end{split}
 \end{equation}
Here, $a_{n,p}$ is the sharp constant from the Aubin-Talenti $L^p$ Sobolev inequality, which we recall below in \eqref{eq:sobolev_cons}. There is an equality in \eqref{eq:LYZAffineSobolev} when $f$ is of the form $$f(v) = \left(\alpha + |A(v-v_0)|^{\frac{p}{p-1}}\right)^{-\frac{n-p}{p}},$$ with $A$ a non-singular $n\times n$ matrix, $\alpha >0$, and $v_0 \in \Rn$. 

Like with Petty's projection inequality \eqref{eq:PettyClassical} and Zhang's affine Sobolev inequality \eqref{eq:ZhangAffineSob}, the inequality \eqref{eq:LYZAffineSobolev} is invariant under affine transforms, and immediately implies the classical $L^p$ Sobolev inequality. Asymmetric versions of the results in \cite{LYZ00,LYZ02} were studied in \cite{HS2009} and \cite{HS09}, respectively, by Haberl and  Schuster. Other important affine $L^p$ Sobolev type inequalities, which are of particular interest, have appeared in \cite{AFTL97, BW93, CENV04, DHJM18, EJ88, HJM16,HJS21,HL22,KS21,  LXZ11,LYZ06,NVH16, XJ07}.  For further information concerning the history of the inequalities \eqref{eq:PettyClassical}, \eqref{eq:ZhangAffineSob} and \eqref{eq:LYZAffineSobolev}, and other Sobolev type inequalities, see e.g., the monographs \cite[Chapter~1]{AGA2} and \cite[Chapter~9]{Sh1} and the myriad of references therein.

In \cite{HLPRY23}, the present authors studied variants of Petty's projection inequality \eqref{eq:PettyClassical} and Zhang's affine Sobolev inequality \eqref{eq:ZhangAffineSob} in the so-called $m$th-order setting. The goal of this article is to continue the study of \cite{HLPRY23}. The general idea is the following. Given a convex set $Q$, its support function is given by $$h_Q(x)=\sup_{y\in Q}\langle y,x \rangle.
$$ Notice that we can write, for $\alpha\in\R$, $|\alpha|=h_{[-1,1]}(\alpha)$. In particular, this means that $|\langle v,\theta\rangle|=h_{[-1,1]}(\langle v,\theta \rangle)$. When viewed from this perspective, there is nothing specific concerning $[-1,1]$, and we replace it with a generic convex body $Q$ from $\R^m$ containing the origin; the results from \cite{HLPRY23,LPRY23,LX24} are precisely the case when $Q$ is the $m$-dimensional orthogonal simplex. In our framework, quantities of the form $\langle v,\theta \rangle$ can and will be replaced by an $m$-tuple of the form $(\langle v,\theta_i \rangle)_{i=1}^m$, where $(\theta_1,\dots,\theta_m)\in\R^{nm}$. As a consequence of our results, we will show the following $m$th-order $L^p$ Sobolev inequality.
In Section~\ref{sec:functionspaces}, we elaborate on the definition of the Sobolev space $W_0^{1,p}(\Rn)$.

\begin{theorem}\label{t:GeneralAffineSobolev} Fix $m,n \in \N,$ $p \in [1,n)$, and a $m$-dimensional convex body $Q$ containing the origin. Set $p^* = \frac{pn}{n-p}$. Consider a non-constant function $f\in W_0^{1,p}(\Rn)$. Then:
\begin{equation*}
d_{n,p}(Q)\left(\int_{\S} \left( \int_{\Rn} h_Q((\nabla f(z))^t.\theta)^pdz \right)^{-\frac{nm} p } d\theta\right)^{-\frac 1 {nm}} \geq a_{n,p}\|f\|_{p^*},
\end{equation*}
where 
$d_{n,p}(Q)$ is a sharp constant given in \eqref{eq:sharp} below. If $p>1$, then there is equality if and only if there exist $\alpha>0$, $A\in GL_n(\R)$ and $v_0\in \R^n$ such that
\[
f(v)=(\alpha+|A.(v-v_0)|^{\frac p {p-1}})^{-\frac{n-p} p }. \]
If $p=1$, then the inequality can be extended to functions of bounded variation (see Lemma~\ref{l:asob}), in which case equality holds if and only if $f$ is a multiple of $\chi_E$ for some $n$-dimensional ellipsoid $E$.
\end{theorem}
\noindent To establish Theorem~\ref{t:GeneralAffineSobolev}, we will generalize a series of $L^p$ inequalities from convex geometry. In fact, the first result of this paper is an $m$th-order extension of the classical inequality \eqref{eq:PettyClassical} and its $L^p$ variants from \cite{HS09,LYZ00} (see Proposition~\ref{p:HS} below). 

As a starting point, it will be natural to state our theorems in the context of convex bodies in matrix spaces.
We consider the vector space $\M$ of real matrices of $n$ rows and $m$ columns, with the usual linear and Euclidean structure.
The Euclidean norm and inner product are given respectively by
\[\langle A,B \rangle = tr(A^t. B) = \sum_{i=1}^n \sum_{j=1}^m a_{i,j} b_{i,j}\ \text{and} \ \ \|A\|_2 = \sqrt{\sum_{i=1}^n \sum_{j=1}^m a_{i,j}^2}\]
for $A,B \in \M$. Here $A^t$,  $tr(A)$ and $a_{i, j}$ denote, respectively, the transpose, trace, and the $(i, j)$th entry of $A$. The matrix multiplication of two matrices $A \in \M$ and $B \in \M[m,k]$ is represented by $A.B \in \M[n,k]$. 

For our first set of results, we have two convex bodies (say $Q$ and $K$). The convex body $Q$ becomes part of an operator, and the convex body $K$ becomes the body being acted on by said operator. For notational convenience, if the convex body $K$ being acted on is from $\R^n$, then we identify $\R^n$ with $\M[n,1]$, that is the body being acted on is a collection of column vectors. Similarly, if the convex body $Q$ that becomes a component of the operator being used is from $\R^m$, then we identify $\R^m$ with $\M[1,m]$. That is, the body ``doing the acting'' is a collection of row vectors. As we will see, the image of $K$ under the action of the operator containing $Q$ will be a convex body in $\M$ (which can be identified with $\R^{nm}$). Throughout, we identify the unit Euclidean spheres $\s \subseteq \R^n = \M[n,1]$, and $\mathbb{S}^{nm-1} \subseteq \M$.

 For a set $G \subseteq \M$, $A \in \M[k,n]$,  and  $B \in \M[m,l]$, we write 
\[A.G = \{A.x \in \M[k,m] : x \in G\} \ \ \mathrm{and}\ \ \ G.B = \{x.B \in \M[n,l] : x \in G\}.\]

\noindent We write $\conbod[k,l]$ for the set of convex bodies in $\M[k,l]$. Similarly, we write $\conbodo[k,l]$ for the collection of all such convex bodies which contain the origin and $\conbodio[k,l]$ for those that contain the origin in their interiors. The polar body of $G\in\conbodio[k,l]$ is
$$G^\circ=\left\{x\in\M[k,l]:h_G(x)\leq 1\right\}.$$
Notice that under these conventions, $G^\circ$ is also in $\M[k,l]$. Additionally,
\begin{equation}
\label{eq:sup_mat}
h_K (A.x) = h_{A^t. K}(x) \quad \text{and} \quad h_K (y.B) = h_{K.B^t}(y)\end{equation} for all $x\in \M[k,m], y \in \M[n,l]$, $K\in \conbod$, $A\in \M[n, k]$ and $B\in \M[l,m]$.

The Lebesgue measure in $\M$ is inherited from the natural identification between $\M$ and $\R^{nm}$. If $D \subset \M[n,m]$ is Lebesgue measurable, $A \in \M[n,n]$, and $ B \in \M[m,m]$, then $$\Vol(A.D) = |\det(A)|^m \Vol(D) \ \ \mathrm{and} \Vol(D.B) = |\det(B)|^n \Vol(D),$$ where $\det (A)$ is the determinant of $A$. 
For $G\in \conbodio$, its gauge or \textit{Minkowski functional} is given by $\|x\|_G=h_{G^\circ}(x)$ for $x\in \M[n,m]$. We have the properties that
\[\|A.x\|_G = \|x\|_{A^{-1}.G} \ \ \mathrm{and} \ \  \|x.B\|_G = \|x\|_{G.B^{-1}}\]
for invertible $A\in \M[n,n]$ and $B\in \M[m,m]$.

We recall that the \textit{$L^p$ surface area measure} of $K\in\conbodio[n,1]$ is given by $${d\sigma_{K,p}(u)=h^{1-p}_K(u) d\sigma_K(u)},$$ where $\sigma_K$ is the surface area measure of $K$ (see Section~\ref{sec:geo_not} for the definition). We now define a more general $L^p$ version of the $m$th-order projection body operator:

\begin{definition} \label{d:generalprojectionbody} Let $p \geq 1$, $m \in \N$, and fix $Q \in \conbodo[1,m]$. Given $K \in \conbodo[n,1]$ with the property that $\sigma_{K,p}$ is a finite Borel measure on $\s$ (e.g. $K\in\conbodio[n,1]$), we define the $(L^p, Q)$-projection body of $K$, $\P K$, via the support function 
\[
h_{\P K}(x)^p = \int_{\s} h_Q(v^t.x)^p d\sigma_{K,p}(v)
\]
for $x\in\M$.
\end{definition}
\noindent 
We establish in Proposition~\ref{p:convexbodyorigin} below that $\P K\in \conbodio$.

We briefly discuss how $\P K$ incorporates other generalizations of the projection operator. By setting $p=1$ and $Q=[-\frac{1}{2},\frac{1}{2}]$, one obtains the classical projection operator. By setting $Q=[-\alpha_1,\alpha_2]$ with $\alpha_1,\alpha_2 > 0$ when $p>1$, we obtain the asymmetric $L^p$ projection bodies by Ludwig \cite{ML05}.
For $m=2,p=1, K \subseteq \R^{2n}$, identifying $\R^2$ with $\mathbb C$ and $\R^{2n}$ with $\mathbb C^n$ yields that a $2n$ dimensional section of $\Pi_{Q,1} K$ coincides with the complex projection body $\Pi_{Q}^{\mathbb C} K$ defined by Abardia and Bernig \cite{AB11}. Namely, for $x \in \mathbb C^n$,
\[h_{\Pi_{Q,1} K}(x,i x) = h_{\Pi_{Q}^{\mathbb C} K}(x).\]

If $K$ is a polytope with facets $F_1,\dots,F_k$ and outer unit-normals $u_1,\dots,u_k$, then, let $\alpha_i=\vol[n-1](F_i)h_K(u_i)^{1-p}$. By using the first formula in \eqref{eq:sup_mat}, we see that
$$h_{\P K}(x)^p = \sum_{i=1}^N \alpha_i h_{v_i.Q}(x)^p,$$
i.e. $\P K = (\alpha_1\cdot v_1.Q) +_p\dots +_p (\alpha_k\cdot v_k.Q)$
(notice that $v_i.Q$ is a linear embedding of $Q$ inside $\M$).
We recall that $+_p$ is Firey's $L^p$ Minkowski summation \cite{Firey62}, see Section~\ref{sec:geo_not}.

Given a fixed non-zero $x\in\M$, we can use the $L^p$ mixed volumes, $V_{p,n}(K,L)$, defined in \eqref{eq:LPmixed} below, to obtain
\begin{equation}
h_{\P K}(x)^p=nV_{p,n}(K,x.Q^t),
\label{eq:bival}
\end{equation}
which follows from the second formula in \eqref{eq:sup_mat} and the fact that $h_{G}(v^t)=h_{G^t}(v)$. We take a moment to observe that, as $x$ varies, the dimension of $x.Q^t$ changes; every dimension from $1$ to $m$ will be obtained in general. This is easily seen by taking $Q$ the orthogonal simplex. In particular, this yields that, when $p=1$, we cannot consider mixed volumes with different indices on $K$. This also yields that we are not, in general, studying projections of $K$ onto subspaces of larger co-dimensions.

We use the natural notation $\PP K:= (\P K)^\circ$ for the $(L^p,Q)$-polar projection body of $K$. In fact, under this notation, the constant from Theorem~\ref{t:GeneralAffineSobolev} is given by
\begin{equation}
\label{eq:sharp}
d_{n,p}(Q) := (n\omega_n)^{\frac{1}{p}}\left(nm\Vol(\PP \B )\right)^\frac 1 {nm}.
\end{equation}
We will show in Proposition~\ref{p:affine_invar_Zhang} below that the functional $$K\in\conbodio[n,1] \mapsto \vol(K)^{\frac{nm} p -m}\Vol(\PP K )$$ is invariant under affine transformations (when $p=1$) or linear transformations (when $p>1$). We thus establish the following $m$th-order $L^p$ Petty projection inequality:

\begin{theorem}\label{t:PPIGeneral} Let $m \in \N$ and $p \geq 1$. Then, for any pair of convex bodies $K \in \conbodio[n,1]$ and $Q \in \conbodo[1,m]$, one has 
\begin{equation}\label{e:PPIGeneral}
\Vol(\PP K ) \vol(K)^{\frac{nm} p -m} \leq \Vol(\PP \B ) \vol(\B)^{\frac{nm} p -m}.
\end{equation}
If $p=1$, then there is equality above if and only if $K$ is an ellipsoid; if $p >1$, then there is equality above if and only if $K$ is an origin symmetric ellipsoid. 
\end{theorem}
\noindent In order to establish the equality conditions of Theorem~\ref{t:PPIGeneral}, we show in Proposition~\ref{p:one_dim_proj} below that for every convex body $Q$, there exists a projection $P(Q)$ of $Q$ onto a one-dimensional linear subspace such that $P(Q)\subseteq Q$. For symmetric $Q$ one may just take the orthogonal projection onto a diameter. For general $Q$ the projection is not necessarily orthogonal, and the proof involves topological methods. We could not find a reference to this fact in the literature and believe this result may be of independent interest. In Section~\ref{sec:connect}, we discuss the relations between Theorem~\ref{t:PPIGeneral} and other Petty projection inequalities. In particular, we emphasize that our result (for $p=1$ and $n$ even) is distinct from the complex case by Haberl \cite{CH19}. We also demonstrate that, in some sense, the result for non-symmetric $Q$ is stronger than for symmetric $Q$.

There has been recent interest in studying extensions of Petty's projection inequality, in addition to the $L^p$ \cite{HS09,LYZ00} setting, such as the mixed \cite{LE85,TS21,LW07}, complex \cite{CH19,WL21,WY21} and Orlicz \cite{BK13, LYZ10} settings and for non-convex sets \cite{YL21,LY22,LX22,TW12, GZ99}. Lutwak conjectured \cite{LE88} a generalization of Petty's projection inequality, where projections onto hyperplanes are replaced by projections onto subspaces of larger co-dimensions. This problem was only recently resolved by Milman and Yehudayoff \cite{MY23}. Petty's projection inequality has also been established in the setting of valuations \cite{BS20, HS19}. Recently,  Henkel and Wannerer \cite{JW24} showed that, if we view the operator $\P K$ as a bivaluation (cf. \cite{ML10}) acting on $\conbodio[n,1]\times\conbodo[1,m]$, then it arises naturally in a framework on reducible bivaluations on $\R^n$.

As a consequence of Theorem~\ref{t:PPIGeneral}, we obtain the following results by sending $p\to\infty$. Consider the gauge function 
\[\|x\|_{E, F}:= \max_{v \in E} \|x.v\|_F, \quad x \in \M[m,n],\] where $E\in\conbodio[n]$  and $F\in\conbodio[m]$. Denote by $B_{E,F} \subseteq \M[m,n]$ its unit ball; notice in the instance that both $E$ and $F$ are symmetric, this quantity is the operator norm of $x \in \M[m,n]$ as an operator between the Banach spaces $(\R^n, \|\cdot\|_E) \to (\R^m, \|\cdot\|_F)$ (by left-multiplication, $v \mapsto x.v$).

\begin{theorem}
\label{t:operatornorms} Let $n,m \in \N$. Then, for any pair of convex bodies $E \in \conbodio[n, 1]$ and $F \in \conbodio[1, m]$, 
\[\Vol(B_{E,F}) \vol[m](F)^{-n} \leq \Vol(B_{E, \B[m]}) \vol[m](\B[m])^{-n},\]
with equality if $F$ is an origin symmetric ellipsoid. Moreover, if $E$ has center of mass at the origin, then
\[\Vol(B_{E,F}) \vol(E)^{m} \leq \Vol(B_{\B, F}) \vol(\B)^{m},\]
with equality if and only if $E$ is an origin symmetric ellipsoid.

In particular, 
\[\left(\frac{\vol(E)^m}{\vol[m](F)^n}\right)\Vol(B_{E,F}) \leq \left(\frac{\vol(\B)^m}{\vol[m](\B[m])^n}\right) \Vol(B_{\B,\B[m]}).\]
\end{theorem}

\noindent Theorem~\ref{t:operatornorms} extends on the case $m=n$ for when $E$ and $F$ are Schatten balls \cite{SR84}, $\ell_p$ balls \cite{CS82}, and Lorentz balls \cite{DV20}.

There are a plethora of affine or linear invariant inequalities of isoperimetric type in convex geometry, among which are the celebrated Busemann-Petty centroid and Santal\'o inequalities. Given a compact set $D \subset \Rn$ with positive volume, its centroid body, $\Gamma D$, is the convex body with support function 
\[
h_{\Gamma D}(\theta) = \frac{1}{\vol(D)} \int_D |\langle x, \theta \rangle| dx, \quad \theta \in \s.
\]
The {\bf Busemann-Petty centroid inequality} asserts that the functional
\begin{equation}\label{eq:BPclassical}
K \in \mathcal{K}^{n,1} \mapsto \vol(\Gamma K) \vol(K)^{-1}
\end{equation}
is minimized precisely when $K$ is an ellipsoid. 

In \cite{LYZ00}, the Busemann-Petty centroid inequality was extended to the $L^p$ setting for all $p > 1$, which turns out to be equivalent to the $L^p$ version of the Petty projection inequality \eqref{eq:PettyClassical} that also appeared in \cite{LYZ00}.  The {\bf $L^p$ Busemann-Petty centroid inequality} asserts that the functional
\begin{equation}\label{eq:LYZBPcentroid}
D\in \mathcal{S}^n \mapsto \vol(\Gamma_p D) \vol(D)^{-1}
\end{equation}
is minimized precisely when $D$ is an origin symmetric ellipsoid.  Here, $\mathcal{S}^n=\mathcal{S}^{n,1}$ denotes the class of star bodies in $\Rn$ (see Section~\ref{sec:geo_not}),  and $\Gamma_p D$ denotes the $L^p$ centroid body of $D$ (see Section~\ref{sec:connect}).  In \cite{HJS21}, the $L^p$ Busemann-Petty centroid inequality \eqref{eq:LYZAffineSobolev}, together with the affine isoperimetric inequalities of \cite{AFTL97} and \cite{CENV04}, was exploited to prove a variety of affine Sobolev type inequalities, among which was the $L^p$ affine Sobolev inequality \eqref{eq:LYZAffineSobolev}.

In \cite[Theorem 1.8]{HLPRY23}, we studied an analogue of the classical Busemann-Petty centroid inequality \eqref{eq:BPclassical} by introducing an $m$th-order extension of the centroid body operator. We now extend this operator to the $L^p$ setting and $Q\in\conbodo[1,m]$. 

\begin{definition}
  \label{d:generalcentroidbody}
  Let $p \geq 1$, $m \in \N$, and fix some $Q \in \conbodo[1,m]$. Given a compact set $L\subset \M$ with positive volume, we define the $(L^p, Q)$-centroid body of $L$, $\G L$, to be the convex body in $\M[n,1]$ with the support function 
  \begin{equation}
  h_{\G L}(v)^p= \frac 1 {\Vol(L)}\int_L h_Q(v^t.x)^p dx.
  \label{eq:cen_hi}
\end{equation}
\end{definition}

We emphasize that, in contrast to the first set of results concerning the $m$th-order projection operator $\P$, the operator $\G$ takes a compact set in $\R^{nm}$ (which we identify with $\M$, that is, this compact set is a collection of $n\times m$ matrices) and produces a convex body in $\R^n$ which we identify with $\M[n,1]$ (that is, $\G L$ is a collection of column vectors). Just like with the projection operator, by setting $p=1$ and $Q=[-1,1]$ in \eqref{eq:cen_hi}, one obtains the classical centroid operator. By setting $Q=[-\alpha_1,\alpha_2]$ with $\alpha_1,\alpha_2 > 0$ when $p>1$, we obtain the asymmetric $L^p$ centroid bodies by Ludwig \cite{ML05}. 
The special case $L = L_1, \times \cdots \times L_m$ with $L_i \in \conbod[n]$ was considered recently in \cite{APPS22}.
The Busemmann-Petty centroid inequality for $\G L$ is the following.

\begin{theorem}
    \label{t:NewLpBPC} Fix $n,m \in \N$. 
Let $L\subset \M[n,m]$ be a compact domain with positive volume, $Q \in \conbodo[1,m]$, and $p \geq 1$. Then
    \begin{equation}\label{eq:NewLpBPC}
     \frac{\vol(\G L)}{\Vol(L)^{1/m}} \geq \frac{\vol(\G \PP \B)}{\Vol(\PP \B)^{1/m}}.
\end{equation}
If $L$ is a star body or a compact domain with piecewise smooth boundary, then there is equality if and only if $L = \PP E$ up to a set of zero volume for some origin symmetric ellipsoid $E \in \conbodio[n,1]$.
\end{theorem}

By setting $Q=[-1,1]$, we obtain the classical result from Busemann and Petty \cite{Busemann53,petty61_1} when $p=1$ and the $L^p$ analogue proven by Lutwak, Yang and Zhang \cite{LYZ00} when $p>1.$ By setting $Q = [-\alpha_1, \alpha_2]$ we obtain the asymmetric $L^p$ case by Haberl and Schuster \cite{HS09}. 

This paper is organized as follows. Section~\ref{sec:geo_not} is dedicated to collecting necessary background materials from convex geometry, Section~\ref{sec:connect} discusses relations between classical $L^p$ Petty projection and Busemann-Petty centroid inequalities and the ones herein established, and Section~\ref{sec:properties} establishes properties of the operators $\P$ and $\G$. In Section~\ref{sec:geometricineqs}, we prove Theorems~\ref{t:PPIGeneral} and \ref{t:NewLpBPC}; Section~\ref{sec:Sant} treats inequalities in the extremal case $p=\infty$, thus establishing Theorem~\ref{t:operatornorms}, and we prove $m$th-order versions of the $L^p$ Santal\'o inequalities by Lutwak and Zhang \cite{LZ97}.  Section~\ref{sec:functionspaces} discusses Sobolev space theory and the convex bodies associated with a Sobolev function. Finally, Section~\ref{sec:functionalineqs} contains the proofs of the Sobolev inequality appearing in Theorem~\ref{t:GeneralAffineSobolev}.

\section{Geometric notions}
\label{sec:geo_not}
Of course, for considerations not involving matrix multiplication, there is no need to distinguish between $\M$ and $\R^{nm}$.
All the results on this section will be stated for $\R^d$ and apply to $\M$ as well. Classical notions we use in this work include $\mathcal O (n),$ the orthogonal group on $\s,$ and $GL_n(\R),$ the group of non-singular $n\times n$ matrices over $\R$.

The following facts about convex bodies can be found in the textbook by Schneider \cite{Sh1}. For $K\in\conbodio[d]$, the \textit{Minkowski functional} of $K$ is defined to be $\|y\|_K=\inf\{r>0:y\in rK\}$. For $K\in\conbodio[d],$ the polar of $K$ is given by $K^\circ = \{x\in\R^d:h_K(x)\leq 1\}$. Notice that $\|x\|_{K^\circ} =h_K(x)$. 

The surface area measure, $\sigma_K$, for a convex body $K \in \conbod[d]$ is a Borel measure on the sphere defined by  
\[\sigma_K(D)=\mathcal H ^{d-1}(n^{-1}_K(D)),\]
for every Borel subset $D$ of $\mathbb{S}^{d-1}$, where $n_K^{-1}(\cdot)$ is the inverse Gauss map of the body $K$. 
The \textit{mixed volume} of $K\in\conbod[d]$ and a compact, convex set $L\subset \R^d$ is given by 
\begin{equation}
	  V_d(K,L):=\frac 1d \lim_{\varepsilon \to 0}\frac{\vol[d](K+\varepsilon L)-\vol[d](K)}{\varepsilon}=\frac 1d \int_{\mathbb{S}^{d-1}}h_L(\theta)d\sigma_K(\theta).
	  \label{eq:mixed_0}
\end{equation}
\noindent We refer to the textbook by Schneider \cite{Sh1} and the references therein for the history on mixed volumes.

For $K,L\in\conbodo[d]$,$\;\alpha,\beta>0$ and $p\geq 1$, Firey defined their $L^p$ Minkowski summation $\alpha\cdot_p K +_p \beta\cdot_p L=\alpha\cdot K+_p \beta\cdot L$ as the convex body whose support function is given by
$$h^p_{\alpha\cdot K+_p \beta\cdot L}(x)=\alpha h_K^p(x)+\beta h_L^p(x),$$
where the first notation emphasizes the fact that $K$ and $L$ are dilated by $\alpha^{1/p}$ and $\beta^{1/p}$ respectively; we will use the second notation, which suppresses this fact. When $p=1$, this corresponds with the standard Minkowski summation of $\alpha K$ and $\beta L$. Recall that Lutwak introduced, and proved, the integral representation of, the so-called $L^p$ mixed volumes \cite{LE93,LE96}: for $p\geq 1$, $K\in\conbodio[d,1]$ and $L\in\conbodo[d,1]$, \begin{equation}
\label{eq:LPmixed}
V_{p,d}(K,L)\!=\!\frac pd \lim_{\varepsilon \to 0}\frac{\vol[d]\!\left(K\!+_p\varepsilon \cdot L\right)\!-\!\vol[d](K)}{\varepsilon}\!=\!\frac 1d\! \int_{\mathbb{S}^{d-1}}\!\!h_L(v)^pd\sigma_{K,p} (v),\end{equation}
where we recall that $d\sigma_{K,p}(v) = h_K(v)^{1-p} d\sigma_K(v)$ is the $L^p$ surface area measure of $K$.
\noindent Lutwak showed the $L^p$ \textit{Minkowski first inequality}: for $K,L\in\conbodio[d]$ and $p\geq 1:$
\begin{equation}
  V_{p,d}(K,L)^d\geq \vol[d](K)^{d-p}\vol[d](L)^p,
  \label{eq:Firey_Min_first}
\end{equation}
with equality if and only if $K$ is a dilate of $L$ when $p>1$, or, if $p=1$, then $K$ is homothetic to $L$.

 We say $L$ is a star body, and write $L\in\mathcal{S}^d$, if it is a compact set containing the origin in its interior such that, for every $x \in L$, $[o,x]\subset L$ and its radial function is continuous and positive on $\R^d\setminus\{o\}$, where the radial function of a compact set $L$ is given by $\rho_L(y)=\sup\{\lambda>0:\lambda y\in L\}.$ For $K\in\conbodio[d]$, we have $\|y\|_K=\rho_K(y)^{-1}$ when $y\neq o$, and we obtain the polar formula for volume:
\begin{equation}
    \label{eq:polar}
    \vol[d](K)=\frac{1}{d}\int_{\mathbb{S}^{d-1}}\rho_K(\theta)^dd\theta.
\end{equation} 
Given $p \in \R$, and $K,L \in \sta[d]$, the $p$th \textit{dual mixed volume} of $K$ and $L$ was introduced by Lutwak in \cite{Lut75} as follows:
\begin{equation}
  \widetilde V _{-p,d}(K,L)=\frac 1d \int_{\mathbb{S}^{d-1}}\rho_K (\theta)^{d+p}\rho_L(\theta)^p d\theta.
  \label{eq:dual_mixed}
\end{equation}
We will require the so-called \textit{dual $L^p$ Minkowski first inequality}: for $K,L \in \mathcal S ^d$, $p\geq 1$
\begin{equation}
  \vol[d](K)^{d+p}\vol[d](L)^{-p}\leq \widetilde V _{-p,d}(K,L)^d ,
  \label{dual_Min_first}
\end{equation}
with equality if and only if $K$ and $L$ are dilates.

 We refer the reader to \cite{gardner_book,Gr,AK05,KY08,Sh1} for more background in convex geometry. 

A key ingredient of the equality case of Theorem \ref{t:PPIGeneral} is to find a linear projection of $Q$ onto a one-dimensional section.
We believe the following result has interest by itself. We use the notation $\xi^\perp$ for the hyperplane through the origin orthogonal to the direction $\xi \in \mathbb{S}^{d-1}$.

\begin{proposition}
\label{p:one_dim_proj}
  Every convex body $Q$ containing the origin admits a linear projection $P$ onto a one-dimensional subspace, for which $P(Q) \subseteq Q$.
\end{proposition}
\begin{proof}
  First we assume $Q$ is smooth with positive Gauss curvature.
  For $\xi \in \mathbb S ^{d-1}$ consider the two points $x_+(\xi), x_-(\xi) \in \partial Q$ for which $n_Q(x_\pm(\xi)) = \pm \xi$.
  The segment $[x_+(\xi), x_-(\xi)] \subseteq Q$ intersects $\xi^\perp$ in a unique point that we call $f(\xi)$.

  The function $f:\mathbb S ^{d-1} \to \R^d$ is even, continuous and represents a vector field on $\mathbb S ^{d-1}$ since always $\xi \perp f(\xi)$.
  We claim that there must be a vector $\xi_0 \in \mathbb S ^{d-1}$ for which $f(\xi_0)=o$.
  Assuming the claim, this implies that the segment $[x_+(\xi), x_-(\xi)]$ passes through the origin, and the projection onto this segment, with kernel $\xi^\perp$, is a linear function taking $Q$ to $[x_+(\xi), x_-(\xi)] \subseteq Q$, as required.

  Now we prove the claim: assume by contradiction that $f$ has no zeros. 
  Let $F:\mathbb{S}^{d-1} \times [0,\frac{\pi} 2 ] \to \mathbb{S}^{d-1}$ be defined by
  \[F(\xi, \alpha) = \frac{\cos(\alpha) \xi + \sin(\alpha) f(\xi)}{|\cos(\alpha)\xi + \sin(\alpha) f(\xi)|}.\]
  Then $F$ is a homotopy between the identity $\mathbb{S}^{d - 1} \to \mathbb{S}^{d - 1}$ and $\tilde f = f/|f|$.
  But $\tilde f$ has even topological degree since it is an even function \cite[Exercise 2.2.14]{HatcherAT}, while the identity map has degree $1$, contradicting the homotopy invariance of the degree. Hence $f$ must have a zero, as claimed.

  To prove the lemma in the general case,
  take $r,R > 0$ such that 
  \begin{equation}
      \label{eq:one_dim_proj}
      B(x_0,r) \subseteq Q \subseteq B(o,R)
  \end{equation}
  for some $x_0 \in Q$, where $B(x,r)=x+r\B$.
  Take a sequence of convex bodies $Q_k$ converging to $Q$ in the Hausdorff metric, which are each smooth with positive Gauss curvature containing the origin, and such that \eqref{eq:one_dim_proj} holds for $Q_k$ instead of $Q$.
  By the first part of the proposition, we may find projections $P_k$ such that $P_k(Q_k)\subseteq Q_k$.
  Notice that we must have 
  \begin{equation}
      \label{eq:one_dim_proj_interior}
      2 B(o,r) \subseteq Q_k - Q_k,
  \end{equation} so that
  \[P_k(2 B(o,r)) \subseteq P_k(Q_k - Q_k) \subseteq Q_k - Q_k \subseteq B(o,2R),\]
 and $\|P\|_{op} \leq \frac Rr$, where $\|\cdot\|_{op}$ is the Euclidean operator norm.
  Since $\{P \in \M[n,n] : \|P\|_{op} \leq \frac Rr \}$ is compact, we may take a subsequence $P_{k_j}$ converging to $P$, which is also a projection, and $P(Q) \subseteq Q$.
\end{proof}

Given $\xi \in \mathbb{S}^{d-1}$ and $K \in \conbod[d]$, the \textit{Steiner symmetral of $K$ about $\xi^{\perp}$}, denoted by $S_\xi K$, is a convex body rearranged from $K$ so that $S_\xi K$ is symmetric with respect to $\xi^{\perp}$ and $\vol[d] (S_{\xi} K) = \vol[d](K)$. It is well known that $K$ can be sent to an origin symmetric Euclidean ball in the Hausdorff metric (see \cite{Sh1}) by successive Steiner symmetrals of $K$ (see \cite[Theorem 6.6.5]{Web94}); that is, one can find $\{\xi_j\}_{j=1}^{\infty} \subset \mathbb{S}^{d-1}$ such that, if we define $S_1 K = S_{\xi_1} K, S_j K=S_{\xi_j} S_{j-1}K,$ then $S_j K \to \vol[d](B_2^d)^{-1/d}\vol[d](K)^{1/d} B_2^d$. It turns out that the Steiner symmetrizaiton process is a powerful tool used to establish many isoperimetric type inequalities. 

\section{Preliminaries on the $m$th-Order $L^p$ Projection and Centroid Operators}

\label{sec:pre_op}

\subsection{Connection between Petty-type inequalities}
\label{sec:connect}
In this section, we briefly discuss how our $m$th-order $L^p$ projection bodies $\P K$, given by Definition~\ref{d:generalprojectionbody} for $p\geq 1$, $Q\in\conbodo[1,m]$ and $K\in\conbodio[n,1]$, and the associated $m$th-order $L^p$ Petty projection inequality, Theorem~\ref{t:PPIGeneral}, are related to other projection bodies and Petty-type inequalities. We will use the matrix notation for the new quantities, and the classical geometric notation for the previous results. Only in this subsection will both notations occur simultaneously. 

Minkowski's formula for projections states that for every $K\in\conbod[d], v\in\mathbb{S}^{d-1},$ one has
\[\vol[d-1](P_{v^\perp} K)=\frac d2 V_d(K,[-v,v]),\]
where $P_{H} K$ denotes the orthogonal projection of $K$ onto a linear subspace $H\subset \R^d$ and in this instance $H=v^{\perp}$. The projection body is defined by the support function
\begin{equation}
\label{eq:proj_sup}
h_{\Pi K}(v)=\vol[d-1](P_{v^\perp} K)=\frac 12 \int_{\mathbb{S}^{d-1}} |\langle u, v\rangle| d\sigma_K(u),\end{equation}
where the second equality comes from \eqref{eq:mixed_0}. The polar projection body of $K$ is denoted by $\Pi^\circ K=(\Pi K)^\circ$ and is the unit ball in $\R^d$ given by the norm $
\|v\|_{\Pi^\circ K} = \vol[d-1](P_{v^\perp} K),$ $v \in \mathbb{S}^{d-1}$. Under these notations, Petty's projection inequality \eqref{eq:PettyClassical} \cite{CMP71} can be written as, for $K\in\conbod[d]$,
\begin{equation}
    \vol[d](K)^{d-1} \vol[d](\Pi^{\circ}K) \leq \left(\frac{\omega_d}{\omega_{d-1}}\right)^d,
    \label{e:petty_ineq}
\end{equation}
with equality if and only if $K$ is an ellipsoid. Therefore, one immediately obtains that Theorem~\ref{t:PPIGeneral} includes the classical  Petty projection inequality \eqref{e:petty_ineq} by setting $p=1$ and $Q=[-\frac{1}{2},\frac{1}{2}]$. 

Recalling that $d\sigma_{K,p}(u)=h^{1-p}_K(u)d\sigma_K(u)$, Lutwak, Yang and Zhang \cite{LYZ00} used $\sigma_{K,p}$ in place of $\sigma_K$ to define the $L^p$ projection body of $K\in\conbodio[d]$, $\Pi_p K,$ via the support function, for $v\in\mathbb{S}^{d-1}$ and $p\geq 1$,
\begin{equation}
  h_{\Pi_p K}(v)^p = \gamma_{d,p} \int_{\mathbb{S}^{d-1}}|\langle u, v \rangle|^p d\sigma_{K,p}(u),
  \label{eq:L^p_proj}
\end{equation}
where $\gamma_{d,p}$ is such that $\Pi_p B^d_2 = B^d_2$. The $L^p$ polar projection body is then naturally $\Pi_p^\circ K:=(\Pi_p K)^\circ.$ The associated Petty-type inequality, proved by Lutwak, Yang and Zhang \cite{LYZ00}, is precisely that the functional
\begin{equation}
\label{eq:LpPetty}
K\mapsto \vol[d](\Pi_p^\circ K ) \vol[d](K)^{\frac{d} p -1}
\end{equation}
over $\conbodio[d]$ is maximized by origin symmetric ellipsoids. Thus, the case $p>1$ and $Q=[-\gamma_{d,p}^{1/p},\gamma_{d,p}^{1/p}]$ of Theorem~\ref{t:PPIGeneral} yields the $L^p$ Petty projection inequality. Ludwig introduced \cite{ML05} the \textit{asymmetric $L^p$ projection bodies} $\Pi^{\pm}_p K$: for $p\geq 1$, $K\in\conbodio[d]$, and $v\in\mathbb{S}^{d-1}$, 
\begin{equation}
\label{eq:asymbod}
	h_{\Pi^{\pm}_p K}(v)^p=\int_{\mathbb{S}^{d-1}} \langle u,v\rangle_{\pm}^p d\sigma_{K,p}(u),
\end{equation}
where we recall that given a function $f:\Rn\to\R$, there exist two non-negative functions $f_\pm$ such that $f=f_+-f_-$. Denoting by $[o,v]$ the line segment connecting the origin to $v\in\R^d\setminus\{o\}$, one has that $h_{[o,\pm v]}(u)=\langle u,v \rangle_\pm.$ For $K\in\conbodio[d]$, $\Pi^{\pm}_p K$ contains the origin in its interior. Indeed, let $R_K$ be the outer-radius of $K$, so that $h_K \leq R_K$. Then, for every $v\in\mathbb{S}^{d-1}$,
\begin{align*}
h_{\Pi^{\pm}_p K}(v)&=\left(\int_{\mathbb{S}^{d-1}}h_{[o,\pm v]}(u)^p h_K(u)^{1-p}d\sigma_K(u)\right)^{\frac{1}{p}}
\\
& \geq R_K^{\frac{1-p}{p}} \vol[d-1](\partial K)^{\frac{1}{p}} \left(\int_{\mathbb{S}^{d-1}}h_{[o,\pm v]}(u)^p \frac{d\sigma_K(u)}{ \vol[d-1](\partial K)}\right)^{\frac{1}{p}}
\\
& \geq (R_K \vol[d-1](\partial K))^{\frac{1-p}{p}} \int_{\mathbb{S}^{d-1}}h_{[o,\pm v]}(u) d\sigma_K(u)
\end{align*}
from Jensen's inequality. But, from the fact that the center of mass of $\sigma_K$ is at the origin,
$$\int_{\mathbb{S}^{d-1}}h_{[o,\pm v]}(u) d\sigma_K(u) = dV_d(K,[o,\pm v]) = \frac{d}{2}V_d(K,[-v,v]).$$
Since $h_{\Pi K}(v) >0$ for all $v\in\s$ as $K$ is full-dimensional and mixed volumes are translation invariant and monotone, we obtain $$h_{\Pi^{\pm}_p K}(v) \geq (R_K \vol[d-1](\partial K))^{\frac{1-p}{p}} h_{\Pi K}(v) >0.$$ Using this, we establish at the end of this subsection  that $\P K$ contains the origin in its interior as well.

Consider constants $\alpha_1,\alpha_2\geq 0$ such that not both are zero, and define the following convex body
\begin{equation}\Pi_{\alpha_1,\alpha_2,p} K = \alpha_1\Pi^+_pK + \alpha_2\Pi^-_pK.
\label{eq:HS_body}
\end{equation}
Using the notation $\Pi_{\alpha_1,\alpha_2,p}^\circ K = (\Pi_{\alpha_1,\alpha_2,p} K)^\circ,$
Haberl and Schuster \cite{HS09} proved the following.
\begin{proposition}[Theorem 1 in \cite{HS09}]
\label{p:HS}
	Let $K\in\conbodio[d]$ and $p>1.$ Then, for $\alpha_1,\alpha_2\geq 0$ such that not both are zero, one has
$$\vol[d](K)^{d/p-1}\vol[d](\Pi_{\alpha_1,\alpha_2,p}^\circ K) \leq \vol[d](B_2^d)^{d/p-1}\vol[d](\Pi_{\alpha_1,\alpha_2,p}^\circ B_2^d),$$
with equality if and only if $K$ is an origin symmetric ellipsoid.
\end{proposition} 
\noindent  By setting $Q=[-\alpha_1,\alpha_2]$ with $\alpha_1,\alpha_2 > 0$ when $p>1$ in Theorem~\ref{t:PPIGeneral}, one obtains Proposition~\ref{p:HS}.

Certain special cases of Theorem~\ref{t:PPIGeneral} are implied by the $L^p$ Petty projection inequality via the product structure. Indeed, if $Q_1, Q_2 \in \conbodo[1,m]$ and $K\in\conbodio[n,1]$, one immediately obtains from Definition \ref{d:generalprojectionbody} that
\begin{equation}
    \label{eq_Qsum}
   \P[Q_1+_p Q_2]K = \P[Q_1]K +_p \P[Q_2]K.
\end{equation}
From \eqref{eq_Qsum} and the dual $L^p$ Brunn-Minkowski inequality (cf. \cite{WF74}), we may deduce Theorem \ref{t:PPIGeneral} for $Q = Q_1+_p Q_2$, from the corresponding result for $Q = Q_1$ and $Q = Q_2$. This explains, by taking $Q_2 = -Q_1$, the phenomenon repeatedly observed in \cite{HS2009, HS09, ML05} that the asymmetric inequalities are stronger and directly imply the symmetric ones.

Also, if $Q = Q_1 \times \{0\}^{m_2}$  with $Q_1 \in \conbod[m_1], m_1 + m_2 = m$, 
\begin{equation}
    \label{eq_Qproduct}
\P K = \P[Q_1] K \times \{0\}^{m_2}.
\end{equation}
Combining \eqref{eq_Qsum} and \eqref{eq_Qproduct} we get 
\[\P[(Q_1 \times \{0\}^{m_2}) +_p (\{0\}^{m_1} \times Q_2)] K = (\P[Q_1]K \times \{0\}^{m_2}) +_p (\{0\}^{m_1} \times \P[Q_2]K)\]
where $Q_i \in \conbodo[m_i], m_1+m_2 = m$, which together with the fact that there exists a constant $c$, depending only on $m_1,m_2$ and $p$, such that
\[\vol[m](((G_1 \times \{0\}^{m_2}) +_p (\{0\}^{m_1} \times G_2))^\circ) = c \vol[m_1](G_1^\circ) \vol[m_2](G_2^\circ)\]
for every $G_i \in \conbodio[m_i]$, shows that Theorem \ref{t:PPIGeneral} for $Q = (Q_1 \times \{0\}^{m_2}) +_p (\{0\}^{m_1} \times Q_2)$ can be deduced directly from the lower dimensional cases $Q=Q_1$ and $Q=Q_2$. For $Q = B_q$, the $\ell_q$ ball with $\frac 1q+\frac 1p=1$, we can repeat this argument $m-1$ times ($B_q$ is a $p$-sum of $m$ orthogonal segments) and reduce Theorem \ref{t:PPIGeneral} to the classical $L^p$ Petty projection inequality \eqref{eq:LpPetty}.

We introduce the notation $o_{k,l}$ for the origin in $\M[k,l]$. If $k$ or $l$ are $1$, we may write $o_l$ or $o_k$ respectively. If the context is clear, we may suppress the subscripts and continue writing just $o$. Let $\{e_{1,i}\}_{i=1}^m$ be the canonical basis in $\M[1,m]$. The orthogonal simplex in $\conbodo[1,m]$ is then given by $$\Delta_m=\text{conv}\left\{o_{1,m},e_{1,1},\dots,e_{1,m}\right\},$$
where conv denotes the closed convex hull operation. If $x\in\M$ is given by $x=[x_1 \, \dots \, x_m],$ where each $x_i$ is a column vector with $n$-entries,
then one obtains from \eqref{eq:bival} that, for $K\in\conbod[n,1]$,
$$h_{\Pi_{(-\Delta_m),1} K}(x)=nV_n(K, \text{conv}_{1\leq i \leq m}[o,-x_i]),$$
where, for brevity, $\cong$ denotes the equating of matrix notation and the standard notation under the Euclidean structure and
$$x.(-\Delta_m)^t=\text{conv}\left\{o_{n,1}^t,-x_1,\dots,-x_m\right\}\cong \text{conv}_{1\leq i \leq m}[o,-x_i].$$ This is precisely the projection body considered in the prequel work \cite[Theorem 1.5]{HLPRY23}, and, when $m=1$ is again the classical projection body from the fact that the center of mass of $\sigma_K$ is at the origin. Thus, Theorem~\ref{t:PPIGeneral} includes the Petty-type inequality from \cite{HLPRY23} by setting $Q$ to be $-\Delta_m$ (notice that $Q$ is not required to have the origin as an interior point) and $p=1$.

We discussed earlier that the complex projection body $\Pi_{Q}^{\mathbb C} K$ can be identified with a $2n$ dimensional section of $\Pi_{Q,1} K$ when $K$ is able to be viewed as a complex, convex body in $\mathbb{C}^n \cong \R^{2n}$. However, the complex Petty projection inequality \cite[Theorem~1.1]{CH19} shown for $\Pi_Q^{\circ,\mathbb C} K=(\Pi_Q^{\mathbb C} K)^\circ$ by Haberl relies on two important facts valid in the case $p=1, m=2$, and even real dimension. The first one is that every symmetric planar convex body $Q$ is a zonoid, and the second one is the existence of a linear embedding $T:\R^n \to \M[n,m]$ for which $T^t.T \equiv I_m$. These facts allow for the volume of $\Pi_{Q}^{\circ, \mathbb C} K$ to be controlled by $\vol(\Pi^\circ K)$, thus reducing the problem to the classical Petty projection inequality.
Here we take a substantially different approach, and, in fact, the complex Petty projection inequality is completely distinct from Theorem~\ref{t:PPIGeneral}, since that inequality controls only a $2n$ dimensional section; in particular, these sections are not necessarily invariant under the $m$th-order Steiner symmetrization used in Section~\ref{sec:geometricineqs} to prove Theorem~\ref{t:PPIGeneral}.

Finally, for every analogue between $\P K$ and some classical projection operator, there is an analogue between $\G K$ and some classical centroid operator. For example, in the prequel to this work, the authors introduced the centroid operator $\Gamma^m $, which assigns to a compact set $D \subset \R^{nm}$ of positive volume a convex body $\Gamma^m D$ in $\R^n$. In the current framework, the operators are related via $\Gamma^m= \Gamma_{-\Delta_m,1}$. We then obtain that the Busemann-Petty centroid inequality from \cite[Theorem 1.8]{HLPRY23} is a special case of Theorem~\ref{t:NewLpBPC}. As another example, the $L^p$ centroid body of $L\in\mathcal{S}^d$, introduced by Lutwak and Zhang \cite{LZ97}, is given by, for $p\geq 1$,
$$h_{\Gamma_p L}(v)^p=\frac{\tilde{\gamma}_{d,p}}{\vol[d](L)}\int_L |\langle v,x \rangle|^p dx,$$
where $\tilde{\gamma}_{d,p}$ is so that $\Gamma_p E = E$ for every origin symmetric ellipsoid $E\in\conbodio[d]$. That is, $\Gamma_p L = \Gamma_{[-\tilde{\gamma}_{d,p}^{1/p},\tilde{\gamma}_{d,p}^{1.p}],p} L$, and the $L^p$ Busemann-Petty centroid inequality by Lutwak, Yang and Zhang \cite{LYZ00} is a particular case of Theorem~\ref{t:NewLpBPC}. 

  We conclude this subsection by showing that $\P K$ and $\G L$ contain the origin (of the appropriate dimension) in their interiors.

\begin{proposition}\label{p:convexbodyorigin}
    For $p\geq 1$, $Q \in \conbodo[1,m]$, $K \in \conbodio[n,1]$ and $L\subset \M$ such that $L$ is a compact set with positive volume containing the origin in its interior, one has $\P K \in \conbodio$ and $\G L\in\conbodio[n,1]$.
\end{proposition}
\begin{proof}
The fact that $h_{\P K}(x)$ in Definition~\ref{d:generalprojectionbody} is convex and bounded for $x \in \S$ is immediate. Similarly, $h_{\G L}(u)$ defined in Definition~\ref{d:generalcentroidbody} is also manifestly convex and bounded for $v\in\s$.

    We first show that $h_{\P K}(x) > 0$ for all $x \in \S$. 
    Since $Q$ has non-empty interior we may find $q \in Q$ for which $x.q^t \neq 0$. Observe that
    \begin{align*}
        h_{\P K}(x)^p &= \int_{\s} h_Q(u^t.x)^p d \sigma_{K,p}(u) \\
        &= \int_{\s} h_{x.Q^t}(u)^p d \sigma_{K,p}(u) \\
        &\geq \int_{\s} h_{[o, x.q^t]}(u)^p d \sigma_{K,p}(u) \\
        &= h_{\Pi_p^+ K}(x.q^t) > 0,
    \end{align*}
    where the last line follows from \eqref{eq:asymbod} and the fact that $\Pi_p^+ K$ contains the origin in its interior. This establishes that $\P K$ has the origin in its interior as well.
    
    We next show that $h_{\G L}(v) >0$ for every $v\in\s$. Take any nonzero $q\in Q$. Since $L$ contains $o_{n,m}$ in its interior, there exists $\varepsilon > 0$ with $\varepsilon B_2^{n m} \subseteq L$. Thus,
    \begin{align}
        \Vol(L) h_{\G L}(v)^p
        &= \int_L h_Q(v^t.x)^p d x \geq \int_{\varepsilon B_2^{n m}} h_{v.[0,q]}(x)^p d x \\
        &= \int_{\varepsilon B_2^{n m}} \langle v.q, x\rangle_+^p d x = \beta_{n,m,p} |v.q| = \beta_{n,m,p} |v| |q| >0,
    \end{align}
    where $\beta_{n,m,p}=\int_{\varepsilon B_2^{n m}} \langle e, x\rangle_+^p d x,$ $e=\frac{v.q}{|v.q|} \in \S$. We note that $\beta_{n,m,p}$ is strictly greater than zero
. The claim follows.
\end{proof}

\subsection{Properties of the $m$th-order $L^p$ projection and centroid bodies} \label{sec:properties}

We start this section by showing some properties concerning $\G$. Our first step is to determine the behavior of the $m$th-order $L^p$ centroid body under matrix multiplication.

\begin{proposition}
    \label{p:linearbehaviorCentr}
	Let $T\in GL_n(\R)$. For $m \in \mathbb N $ and a compact set $L\subset \M$ with positive volume, one has
  \begin{equation}
    \G T .L = T.\G L . 
  \end{equation}
\end{proposition}
\begin{proof}
	The result follows from the definition. Indeed, for $v \in \M[n,1]$,
  \begin{align*}
  h_{\G T. L }(v)^p&=\frac 1 {\Vol(T.L)} \int_{T.L} h_Q(v^t .  x )^p d x \\
  &=\frac 1 {\Vol(L)} \int_L h_Q(v^t . T .  x )^p d x \\
	  &=\frac 1 {\Vol(L)} \int_L h_Q((T^t . v)^t .  x )^p d x \\
  &=h_{\G L }(T^t.v)^p=h_{T.\G L }(v)^p.
  \end{align*}
\end{proof}
We next establish the following duality property between $\G $ and $\PP .$

\begin{lemma}
	\label{l:dualityofmixedvolume} 
	Fix $p\geq 1$ and $K \in \conbodio[n,1]$. Then, for every $Q\in\conbodo[1,m]$ and $L \in \sta$, we have 
	\[
		\widetilde V _{-p,nm}(L,\PP K ) = \frac{(nm+p) \Vol(L)} m V_{p,n}(K, \G L ).
    \]
\end{lemma}

\begin{proof} Observe that 
\begin{align*}
&V_{p,n}(K, \G L )
	= \frac 1n \int_{\s}h_{\G L }(\xi)^p d\sigma_{K,p}(\xi)\\
	&=\frac 1 {n\Vol(L)} \int_{\s} \int_L h_Q(\xi^t . x )^p d x d\sigma_{K,p}(\xi)\\
	&= \frac 1 {n\Vol(L)} \int_{\s} \int_{\S} \int_0^{\rho_L(\theta)} t^{nm+p-1} h_Q(\xi^t . \theta )^p dt d\theta d\sigma_{K,p}(\xi)\\
	&= \frac 1 {n(nm+p)\Vol(L)} \int_{\S}\rho_L(\theta)^{nm+p} \int_{\s} h_Q(\xi^t . \theta )^p d\sigma_{K,p}(\xi) d\theta\\
	&= \frac 1 {n(nm+p)\Vol(L)} \int_{\S} \rho_L (\theta)^{nm+p} \rho_{\PP K }(\theta)^{-p}d\theta\\
	&=\frac m {(nm+p)\Vol(L)} \widetilde V _{-p,nm}(L,\PP K ), 
\end{align*}
as required. 
\end{proof}

An immediate result of this duality is the following identity, by setting $K=\G L$.
\begin{proposition}
\label{p:Gvolume}
Fix $p\geq 1$ and $Q\in\conbodo[1,m]$. Then, for every $L \in \sta$, 
	\[
		\vol(\G L)=\frac{m\widetilde V _{-p,nm}(L,\PP \G L )}{(nm+p) \Vol(L)}.
    \]
\end{proposition}

Using the duality property from Lemma~\ref{l:dualityofmixedvolume}, we obtain how $\PP$ behaves under linear transformations.

\begin{proposition}
	\label{p:linearbehaviorProj}
	Fix $p\geq 1$ and $K \in \conbodio[n,1]$. Then, for every $Q \in \conbodo[1,m]$, $A \in GL_n(\R),$ and $B \in \M[m,l]$, one has 
	\begin{align}
 \label{eq:linearbehaviorProjLeft}
	\PP {(A.K)} &= |\det(A)|^{-\frac 1p} A .\PP K, \\
 \label{eq:linearbehaviorProjRight}
	\P[Q.B] K &= (\P K ).B.    
	\end{align} 
\end{proposition}

\begin{proof} We first consider the case when $A\in SL_n(\R).$ Observe that, according to Lemma~\ref{l:dualityofmixedvolume}, for any star body $L\in \sta$, using properties of mixed volume, dual mixed volumes, and Proposition~\ref{p:linearbehaviorCentr}, we have 
\begin{align*}
\left(\frac m {nm +p} \right)\left(\frac{\widetilde V _{-p,nm}(L,\PP A.K )}{\Vol(L)} \right)
  &= V_{p,n}(A.K,\G L )\\
  &= V_{p,n}(K, A^{-1}.\G L )\\
  &= V_{p,n}(K,\G {A^{-1}}.L )\\
  &=\left(\frac m {nm +p} \right)\left(\frac{\widetilde V _{-p,nm}({A^{-1}}.L,\PP K)}{\Vol({A^{-1}}.L)}\right)\\
  &= \left(\frac m {nm +p} \right)\left(\frac{\widetilde V _{-p,nm}(L, A .\PP K)}{\Vol(L)}\right).
\end{align*} Therefore, we have shown that 
\[
\widetilde V _{-p,nm}(L,\PP A.K) = \widetilde V _{-p,nm}(L, A .\PP K)
\]
holds for all $L \in \sta$. Consequently,
\[
\PP A.K = A .\PP K .
\]
The general case $A \in GL_n(\R)$ follows from the homogeneity, $\P{aK}=a^{\frac{n-p} p }\P K , a>0$.

To prove the second statement, for every $x \in \M[n,l],$
\begin{align}
h_{(\P K).B}(x)^p
&= h_{\P K}(x.B^t)^p  \\
&= \int_{\s} h_Q( v^t. x .B^t )^p d \sigma_{K,p}(v) \\
	&= \int_{\s} h_{Q .B}( v^t. x)^p d \sigma_{K,p}(v) \\
&= h_{\P [Q.B]}(x)^p.
\end{align}
\end{proof}

We now study the associated invariant quantity for $\PP $.
\begin{proposition}[The $m$th-order $L^p$ Petty product]
\label{p:affine_invar_Zhang}
	For $m\in\mathbb N $ and $Q\in\conbodo[1,m],$ the following functional is invariant under linear transformation when $p>1$ and affine transformations when $p=1$:
  $$K\in\conbodio[n,1] \mapsto \vol(K)^{\frac{nm} p -m}\Vol(\PP K).$$
\end{proposition}
\begin{proof}
  When $p=1$, we can use the fact that both volume, as a functional on $\M[n,1],$ and the
  surface area measure are translation invariant to reduce to the case of linear transformations. Then, for all $p \geq 1,$ the claim is immediate from Proposition~\ref{p:linearbehaviorProj}.
\end{proof}

\begin{proposition}[Continuity of $\P $] \label{p:continuityofmultdimproj}
  Let $m \in \N$, $p \geq 1$, and $Q \in \conbodo[1,m]$. 
  Let $\{K_j\}_j \subset \conbodio[n,1]$ be a sequence of convex bodies such that $K_j\to K$ with respect to the Hausdorff metric on $\conbodo[n,1]$, where $K\in\conbodo[n,1]$ has the property that $\sigma_{K,p}$ is a finite Borel measure on $\s$. Then one has $\P K_j \to \P K $ with respect to the Hausdorff metric on $\conbodio.$
\end{proposition}
\begin{proof} The statement that $\P K_j \!\to\! \P K $ in the Hausdorff metric is equivalent to the claim that $h_{\P K_j } \!\to\! h_{\P K }$ uniformly on $\S$. One easily verifies that the map $\S \!\to\! \conbodo[n,1]$ defined by $x \mapsto x.Q^t$ is continuous with respect to the Hausdorff metric (and we use this since we recall, for every $v\in\M[n,1]$, $h_Q(v^t.x)\!=\!h_{x.Q^t}(v)$). Additionally, the $L^p$ mixed volumes are continuous in the Hausdorff metric \cite{LE93}. Combining these two facts, one deduces that the map $\conbodo[n,1] \!\times\! \S \!\to\! \mathbb R^+$ defined by $(K, x) \!\to\! h_{\P K }(x)$ is continuous when restricted to the subset of $\conbodo[n,1]$ consisting of those $K$ whose $L^p$ surface area measure is finite. Since $\S$ is compact,  the map $\conbodio[n,1] \!\!\to\! C(\S)$ defined by $K\! \mapsto\! h_{\P K}$ is continuous, which precisely means that if $K_j \!\to\! K,$ where $K\!\in\conbodo[n,1]$ is a limit point of $\conbodio[n,1]$ with finite $L^p$ surface area measure, then $h_{\P K_j } \!\to\! h_{\P K }$ uniformly on $\S$.
\end{proof}

We remark that in Proposition~\ref{p:continuityofmultdimproj} the assumptions on $K$ can be lessened to $K\in\conbod[n,1]$ when $p=1$. Notice that $\PP:\conbodo[n,1] \to \conbodio$ and $\G :\sta \to \conbodio[n,1]$.
The following lemma shows that the $m$th-order $L^p$ Petty product increases when $K$ is replaced by $\G \PP K$, which shows in particular that a necessary condition for equality in the $m$th-order $L^p$ Petty projection inequality, Theorem~\ref{t:PPIGeneral}, is that $K$ belongs to subclass of $\conbodio[n,1]$ that is invariant under the operator $\G \PP$.
\begin{lemma}
\label{l:invar}
Fix $n,m \in \N$, $p\geq 1$ and $Q\in\conbodo[1,m]$. Then, for every $K\in\conbodio[n,1]$, one obtains
\begin{align*}\vol(K)^{\frac{nm} p -m}&\Vol(\PP K) 
\\
&\leq \vol(\G \PP K)^{\frac{nm} p -m}\Vol(\PP \G \PP K),
\end{align*}
with equality if and only if $K$ is a dilate of $\G \PP K$.
\end{lemma}

The proof of Lemma~\ref{l:invar} is just concatenating Propositions \ref{p:BP_P_rel} and \ref{p:BP_P_rel2} below. 
\begin{proposition}
    \label{p:BP_P_rel}
    Fix $n,m\in\N$, $p\geq 1$ and $Q\in\conbodo[1,m]$. Then, for $L\in\sta$,
\begin{equation}
    \vol(\G L )^{\frac{nm}{p}-m}\Vol(\PP \G L )\geq \left(\frac m {nm+p}\right)^{\frac{nm}{p}} \left(\frac{\vol(\G L )}{\Vol(L)^{1/m}}\right)^{-m}.
\end{equation}
  with equality if and only if $L$ is a dilate of $\PP \G L $.
\end{proposition}
\begin{proof}
    The result is immediate from Proposition~\ref{p:Gvolume} and the dual $L^p$ Minkowski first inequality \eqref{dual_Min_first}.
\end{proof}

\begin{proposition}
\label{p:BP_P_rel2}
    Fix $n,m \in \N$, $p\geq 1$ and $Q\in\conbodo[1,m]$. Then, for every $K\in\conbodio[n,1]$, one obtains
    \[
    \left(\frac{m}{nm+p}\right)^\frac{nm}{p}\left(\frac{\vol(\G \PP K)}{\Vol(\PP K)^{1/m}}\right)^{-m} \geq \vol(K)^{\frac{nm}{p}-m}\Vol(\PP K),
    \]
    with equality if and only if $K$ is a dilate ($p>1$) of, or homothetic to ($p=1$), $\G\PP K$.
\end{proposition}
\begin{proof}
    Set in Lemma~\ref{l:dualityofmixedvolume} that $L=\PP K$. Then, one obtains from the $L^p$ Minkowski first inequality \eqref{eq:Firey_Min_first},
    \begin{align*}
		\left(\frac{m}{nm+p}\right)^\frac{nm}{p} &= V_{p,n}(K, \G \PP K )^\frac{nm}{p} 
  \\
  &\geq \vol(K)^{\frac{nm}{p}-m}\vol(\G \PP K)^{m}.
    \end{align*}
\end{proof}

Finally, we establish that the class of origin symmetric ellipsoids is invariant under $\G \PP$.
\begin{lemma}
\label{l:class}
	Let $Q\in\conbodo[1,m]$ and $p\geq 1$. Let $E$ be an origin symmetric ellipsoid in $\M[n,1]$. Then, 
 \[\G \PP E = \vol(\B)^{\frac{1}{p}}\vol(E)^{-\frac 1 {p}} C_{n,m, p} E,\]
 where
 \begin{equation}
  \label{eq:elipp_cal}
    C_{n,m, p}=\left(\frac m {\vol(\B)(nm+p)}\right)^{\frac 1p }.
  \end{equation}
\end{lemma}
\begin{proof}
First, we show that $\G \PP{\B}$ is rotation invariant, and thus a dilate of $\B$.
	We first notice that, for every $T\in \mathcal O (n)$, Propositions \ref{p:linearbehaviorCentr} and \ref{p:linearbehaviorProj} yield 
 \[\G \PP \B = \G \PP T.\B = \G T.\PP \B= T. \G \PP \B;\]
this means that $\G \PP \B$ is rotation invariant and thus a ball. Next, let $E\in\conbodio[n,1]$ be an origin symmetric ellipsoid. Then, there exists $T\in GL_n(\R)$ so that $E=T. \B.$ From Proposition~\ref{p:linearbehaviorProj}, we have
   $$\PP T .\B = |\det (T)|^{-\frac 1p} T .\PP \B .$$
   Applying $\G,$ we obtain
   \begin{align*}
   \G \PP T. \B 
   &=\G |\det (T)|^{-\frac 1p }T. \PP \B  \\
   &=|\det (T)|^{-\frac 1p }T.\G \PP \B  \\
   &=\vol(\B)^{\frac 1 {p}}\vol(E)^{-\frac 1 {p}}C_{n,m, p} T .\B
   \end{align*}
   for some $C_{n,m,p}>0.$ To establish the formula for $C_{n,m,p}$, set $K=\B$ and $L= \PP \B$ in in Lemma~\ref{l:dualityofmixedvolume} to obtain:
   \begin{align*}
		\frac{m}{(nm+p)} &=  V_{p,n}(\B, \G \PP \B) 
  \\
  &= V_{p,n}(\B, C_{n,m,p} \B)=\vol(\B)C_{n,m,p}^p.
    \end{align*}
\end{proof}

\section{On $m$th-Order Isoperimetric Inequalities}

\subsection{The Petty and Busemann-Petty inequalities} \label{sec:geometricineqs}

Bianchi, Gardner and Gronchi \cite{BGG17} introduced an abstract framework of Steiner symmetrization, expanding the concept of  the Fibre combination of convex bodies introduced by  McMullen \cite{McM99}. A particular case of this framework, isolated in \cite{HLPRY23} by the authors, is the natural analogue of Steiner symmetrization in the higher-order setting defined as follows:

\begin{definition}\label{d:multidimSteinerSymmetrization}

	Fix $m,n \in \N$. For $v \in\s$, consider the $m$-dimensional space $[ v ]:=\{v. t: t \in \M[1,m]\} \subseteq \M[n,m]$ and let $V(v)$ be its orthogonal complement, this is
	\[V(v) = \{x \in \M: v^t.x = o \in \M[1,m]\}.\]
Let $L \subseteq \M$ be a compact set with non-empty interior.
We define the $m$th-order Steiner symmetral of $L$ with respect to $v$
\begin{equation}
	\label{eq:sym_def}
 \begin{split}
	\bar S_v L=
	\bigg\{y+ v .\frac{t-s}2 \in & \M \colon y \in V(v), t,s \in \M[1,m], 
 \\
 &(y + v .t), (y + v. s) \in L \bigg\}.
 \end{split}
\end{equation}
\end{definition}

It was shown in \cite{UJ23} that, for $L\in\conbod$, 
\begin{equation}\vol[nm](L)\leq \vol[nm](\bar S_v L).
\label{eq:vol_not_dec}
\end{equation} Using this symmetrization, in conjunction with techniques established in \cite{LYZ00} and used in \cite{HS09,HLPRY23}, we will establish a set-inclusion to then prove Theorem~\ref{t:PPIGeneral}.

\begin{lemma}
	\label{l:multidimSteiner}
	Fix $v \in \s$, $Q\in\conbodo[1,m],$ and $p\geq 1$.	For $K\in\conbodio[n,1]$,
	\[\bar S_v \PP K \subseteq \PP  S_v K .\]
\end{lemma}

\begin{proof}
By continuity of $\bar S_v$ and $\PP $ with respect to the Hausdorff metric, we may assume $K$ has a $C^1$-smooth boundary and the origin in its interior. Then,
\[
K = \{u+\tau v \in \R^n \colon f(u) \leq \tau \leq g(u), u \in P_{v^{\perp}}K\},
\]
where $f, g \colon v^\perp \to \R$ are $C^1$ functions defined in the relative interior of $P_{v^{\perp}}K$. Then the (classical) Steiner symmetral in the direction $v$ has the form
\[S_v K = \{u+\tau v \in \R^n \colon -z(u) \leq \tau \leq z(u), u \in P_{v^{\perp}}K \},
\]
where $z = \frac 12(g - f)$.

	Set $\theta = x+v. t$ with $x \in V(v), t \in \M[1,m]$.
	Notice that for $u$ in the interior of $P_{v^\perp}K,$ the outer unit normal on $\partial K$ at the points $u+g(u) v$ and $u+f(u) v$ are given by
	\[n_K(u + g (u) v) = \frac{-\nabla g(u) + v}{|-\nabla g(u) + v|} \quad \text{ and } \quad n_K(u + f (u)v) = \frac{\nabla f(u)-v}{|\nabla f(u) - v|},\]
	respectively, where $\nabla g (u), \nabla f(u) \in v^\perp$.
	Applying a change of variables, we obtain 
	\begin{align*}
        &h_{\P K }(\theta)^p 
		= \int_{\s} h_Q(\xi^t.\theta )^p h_K(\xi)^{1-p}d\sigma_K(\xi)\\
		&= \int_{P_{v^{\perp}}K} h_Q(n_K(u+g(u)v)^t .\theta )^p h_K \left(n_K(u+g(u)v)\right)^{1-p}|-\nabla g(u) + v| du \\
		&+ \int_{P_{v^{\perp}}K} h_Q(n_K(u+f(u)v)^t.\theta )^p h_K\left(n_K(u+f(u)v)\right)^{1-p}|\nabla f(u) - v| du.
  \end{align*}
  As in \cite{LYZ00}, we will use the following notation: given a differentiable function $g$ defined on an open set $D \subset v^\perp$, we set 
\[
\langle g \rangle(u) = g(u) - \nabla g(u)^t.u, \quad u \in D,
\]
with the gradient $\nabla g(u)$ belonging to $v^\perp$. Since $\nabla g(u),u\in v^\perp$, we have
\begin{align*}
h_K\left(n_K(u+g(u) v)\right) & =\left(n_K(u+g(u) v)\right)^t (u+g(u) v) \\
& =\frac{(-\nabla g(u)+v)^t (u+g(u) v)}{|-\nabla g(u)+v|} \\
& =\frac{\langle g\rangle(u)}{|-\nabla g(u)+v|},
\end{align*}
(and analogously for $f$). From the $1$-homogeneity of the support function $h_Q$,
\begin{align*}
		h_{\P K }(\theta)^p
		&=\int_{P_{v^{\perp}}K} h_Q((-\nabla g(u) + v)^t. \theta)^p \langle g \rangle(u)^{1-p} du
  \\
		&+\int_{P_{v^{\perp}}K} h_Q((\nabla f(u) - v)^t. \theta)^p \langle - f \rangle(u)^{1-p}du.
\end{align*}
Using the definition of $\theta$, we then obtain
\begin{align*}
h_{\P K }(\theta)^p
		&=\int_{P_{v^{\perp}}K} h_Q( t-\nabla g(u)^t.x )^p \langle g \rangle(u)^{1-p} du
  \\
		&+ \int_{P_{v^{\perp}}K} h_Q( \nabla f(u)^t.x - t )^p \langle -f \rangle (u)^{1-p} du.
	\end{align*}
	Now assume $x+v.r\! \in\! \bar S_v \PP K$ with $x\in V(v), r \in \M[1,m]$. By definition, there are $t,s\! \in\! \M[1,m]$ with $h_{\P K }( x+v .t ), h_{\P K }(x+v. s )\! \leq\! 1$ and $r = \frac{t - s}2$. Making use of the above representation for the support function of $\P K $, 
\begin{align*}
h_{\P S_v K }(x+v .r)^p 
&= \int_{P_{v^{\perp}} K} h_Q(r - \nabla z(u)^t.x)^p [\langle z \rangle(u)]^{1-p}du\\
&+\int_{P_{v^{\perp}} K} h_Q( -\nabla z(u)^t.x -r)^p [\langle z \rangle(u)]^{1-p}du.
\end{align*}
Inserting the definition of $z$ yields
\begin{align*}
&h_{\P S_v K }(x+v. r)^p 
\\
&\!=\int_{P_{v^{\perp}} K} h_Q\left(\frac{1}{2}(t-s) - \frac{1}{2}( \nabla g(u)^t.x - \nabla f(u)^t.x) \right)^p\left[\frac{1}{2}\langle g-f \rangle(u) \right]^{1-p}du\\
&\!+\int_{P_{v^{\perp}} K} h_Q\left(-\frac{1}{2} \left(\nabla g(u)^t.x \!-\! \nabla f(u)^t.x\right)  - \frac{1}{2}(t-s) \right)^p\left[\frac{1}{2}\langle g-f \rangle(u) \right]^{1-p}\!du.
\end{align*}
Using the sub-linearity of $h_Q(\cdot)$, and that $(\alpha+\beta)^p(\gamma+\lambda)^{1-p} \leq \alpha^p\gamma^{1-p} + \beta^p\lambda^{1-p}$, for $\alpha,\gamma,\beta,\lambda >0$ \cite[Lemma 8]{LYZ00}, our computation continues as:
\begin{align*}
&h_{\P S_v K }(x+v. r)^p 
\\
&\leq \frac{1}{2}\int_{P_{v^\perp}K} (h_Q(t\!-\!\nabla g(u)^t.x) + h_Q(\nabla f(u)^t.x \!-\! s ))^p[(\langle g\rangle(u) + \langle\!-\! f \rangle(u))]^{1-p}du\\
&+\frac{1}{2} \int_{P_{v^\perp}K} (h_Q(s-\nabla g(u)^t.x) + h_Q(\nabla f(u)^t.x\!-\!t) )^p[\langle g \rangle(u) + \langle -f \rangle(u)]^{1-p}du.
\end{align*}
Using the sub-linearity property again yields
\begin{align*}
&h_{\P S_v K }(x+v. r)^p 
\\
&\leq \frac{1}{2} \int_{P_{v^\perp}K}\!\left( h_Q(t\!-\!\nabla g(u)^t.x)^p [\langle g \rangle(\!u\!)]^{1-p}+h_Q(\nabla f(u)^t.x\!-\!t)^p [\langle \!-\!f \rangle(\!u\!)]^{1-p}  \right)du\\
&+\frac{1}{2} \int_{P_{v^\perp}K}\!\left( h_Q(s\!-\!\nabla g(u)^t.x)^p [\langle g \rangle(\!u\!)]^{1-p}+h_Q(\nabla f(u)^t.x\!-\!s)^p [\langle \!-\!f \rangle(\!u\!)]^{1-p}  \right)du\\
&= \frac{1}{2}h_{\P K}(x + v.t) + \frac{1}{2} h_{\P K}(x+v.s) \leq 1,
\end{align*}
which means $x + v .r \in \PP S_v K$, completing the proof.
\end{proof}

Let $r \in \M[1,m]$ be a non-zero vector, and define $J_r:\M[n,1] \to \M[n,m]$ by $J_r(u) = u.r$. Below, $J_r^{-1}$ denotes the preimage of $J_r$.

\begin{lemma}
\label{l:permutationofvariables}
Let $Q\in\conbodo[1,m]$, then there exists a non-zero $r=r(Q) \in \M[1,m]$, depending only on $Q$, such that for every
$K\in\conbodio[n,1]$,  $p\geq 1$, and $v\in\s$,
  \[J_r^{-1} \bar S_v \PP K \subseteq S_v \pp[h_Q(r)^p, h_Q(-r)^p] K,\]
	where $\Pi_{\alpha_1,\alpha_2,p} K$ is the body given by \eqref{eq:HS_body}.
\end{lemma}

\begin{proof}
For the given convex body $Q$, let $P=r^t.l \in \M[m,m]$ represent, by right-multiplication, the projection given by Proposition \ref{p:one_dim_proj}, where $r,l \in \M[1,m]$. Thus, $Q.P$ is a one-dimensional linear subspace such that $Q.P \subseteq Q$. Denote $\alpha_1 = h_Q(r)^p, \alpha_2 = h_Q(-r)^p$.
Using \eqref{eq:linearbehaviorProjRight}, one obtains
\begin{align*}
    \P K \supseteq \P[Q.P]K = (\P[Q.r^t] K).l = (\P[{[-h_Q(-r),h_Q(r)]}] K).l.
\end{align*}

Consequently,
  \begin{equation}
    \label{l:permutationofvariables:eq1}
	  (\PP K). l^t \subseteq \pp K.
  \end{equation}
  
  To prove the lemma, take $u+\tau v \in J_r^{-1}(\bar S_v \PP K)$ with $u \perp v \in \R^n, \tau \in \R$.
  By the definition of $\bar S_v$, we can write
	\begin{equation}
	  \label{l:permutationofvariables:eq2}
	(u+\tau v).r = x + v. \frac {t-s}2
     \end{equation}
	with $x\in\M[n,m], t,s \in \M[1,m], (x+t.v), (x+s.v) \in \PP K$.
	Multiplying both sides of \eqref{l:permutationofvariables:eq2} on the left by $v^t$, we find
	\begin{equation}
	  \label{l:permutationofvariables:eq3}
	 \tau r = \frac {t-s}2, \quad u .r = x.
	\end{equation}
	Take $\alpha = t.l^t, \beta = s.l^t \in \R$. Observe that, since $r.l^t=tr(P)=1$,
     \[(x + v. t).l^t = u+\alpha v,\quad (x + v. t).l^t = u+\beta v.\]
	This together with \eqref{l:permutationofvariables:eq1} implies $u+ \alpha v, u+ \beta v \in \pp K$.

	Taking the right-product with $l^t$ at \eqref{l:permutationofvariables:eq3}, we obtain
	\[\tau = \frac{t-s}2.l^t = \frac{\alpha-\beta}2.\]
	We obtain immediately that
	\[x+\tau v = x + \frac{\alpha-\beta}2 v ,\]
	  which implies $u+\tau v \in S_v \pp K$.
\end{proof}

We are now ready to prove Petty's projection inequality in our setting. For the convenience of the reader, we recall it.
\\
\vspace{1mm}
\noindent {\bf Theorem~\ref{t:PPIGeneral} (The $m$th-order $L^p$ Petty's projection inequality).} 
\hfill\break
{\it
	Fix $m \!\in\! \N$ and $p \geq 1$. Then, for any pair of convex bodies $K \in \conbodio[n,1]$ and $Q \in \conbodo[1,m]$, one has 
\begin{equation}
\Vol(\PP K ) \vol(K)^{\frac{nm} p -m} \leq \Vol(\PP \B ) \vol(\B)^{\frac{nm} p -m}.
\end{equation}
If $p=1$ then there is equality above  if and only if $K$ is an ellipsoid; if $p >1$, then there is equality above if and only if $K$ is an origin symmetric ellipsoid.}

\begin{proof}
	First we prove the inequality. Using the fact that the $m$th-order Steiner symmetrization does not decrease the volume of $\PP K $ via \eqref{eq:vol_not_dec} and Lemma~\ref{l:multidimSteiner}, we observe that, for any given $v \in \mathbb S ^{n-1}$,
	\begin{equation}
	  \label{eq:pettychainofinequalities}
   \begin{split}
	  \vol(K)^{\frac{nm} p -m}\Vol(\PP K ) &\leq \vol(S_v K)^{\frac{nm} p -m}\Vol(\bar S_v \PP K ) 
   \\
   &\leq \vol(S_v K)^{\frac{nm} p -m}\Vol(\PP S_v K ).
   \end{split}
	\end{equation}
 We choose a sequence of directions $\{v_j\}_j \subset \s$ such that $$S_j K \to \vol(\B)^{-1/n}\vol(K)^{1/n}\B$$ in the Hausdorff metric.
	By Proposition \ref{p:continuityofmultdimproj}, we obtain the inequality.

	Now assume there is equality in \eqref{eq:pettychainofinequalities}.
  By Lemma \ref{l:multidimSteiner} and the equality of volumes, $\bar S_v \PP K = \PP S_v K $. Let $\alpha_1, \alpha_2$ and $r$ be as given in Lemma \ref{l:permutationofvariables}. Recalling $J_r: \M[n, 1] \rightarrow \M$, one immediately has that, for every $G\in\conbodio[n,1]$, \begin{align*}
 x \in J_r^{-1}\left(\Pi_{Q, p}^{\circ} G\right) &\Leftrightarrow \quad  x. r \in \PP G  
 \\
 \Leftrightarrow h_{\P G}(x . r) \leq 1 &\Leftrightarrow h_{\left(\P G\right) . r}(x) \leq 1 \\
 \Leftrightarrow h_{\Pi_{Q . r^t}, p} G{(x)}\leq 1  &\Leftrightarrow x\in \Pi^\circ_{\alpha_1,\alpha_2,p} G\end{align*} from the definition of $r$. Applying this to $G=S_v K$ and using Lemma \ref{l:permutationofvariables}, we have
	\[\Pi^\circ_{\alpha_1, \alpha_2,p} S_v K = J_r^{-1}(\PP S_v K)= J_r^{-1}(\bar S_v \PP K) \subseteq S_v \Pi^\circ_{\alpha_1, \alpha_2,p} K\]
  and therefore,
 \[ \Pi^\circ_{\alpha_1, \alpha_2,p} S_v K = S_v \Pi^\circ_{\alpha_1, \alpha_2,p} K\]
  for every $v \in \s$.
  For $p>1$ this implies that $K$ is an extremal body for Proposition~\ref{p:HS} (which is \cite[Theorem 1]{HS09}), and thus an origin symmetric ellipsoid.
  For $p=1$ this implies that $K$ is an extremal body for the classical Petty projection inequality \eqref{e:petty_ineq}, and thus, an ellipsoid.
\end{proof}

Consider $Q\in\conbodio[1,m]$. Observe that, for every $x\in\S$,
$$h_{\P K}(x)^p \geq \min_{v \in \text{supp}(\sigma_{K,p})} h_{x.Q^t}(v)^p\sigma_{K,p}(\s).$$
Thus, if $K\in\conbodo[n,1]$ is so that $\sigma_{K,p}(\s)$ is infinite, then $\P K$ is all of $\R^{nm}$, and so $\PP K$ is $\{o_{n,m}\}$. Consequently, Theorem~\ref{t:PPIGeneral} actually holds for all $K\in\conbodo[n,1]$ in this instance. However, this observation shows us that $\Vol(\PP K)\vol(K)^{\frac{nm}{p}-m}$ will sometimes be zero. In particular, an $L^p$ analogue of the so-called Zhang's projection inequality \cite{Zhang91}, which showed that $\vol(K)^{n-1}\vol(\Pi^\circ K)$ is uniquely minimized by the class of simplices, is not possible. If one restricts to just symmetric convex bodies, then the question of Zhang's projection inequality has meaning. We note that even in the case $p=1$, $Q=[-1,1]$, the minimizer of $K\mapsto \vol(K)^{n-1}\vol(\Pi^\circ K)$ among symmetric convex bodies is still unknown.

We now obtain the $m$th-order $L^p$ Busemann-Petty centroid inequality, Theorem~\ref{t:NewLpBPC}, as a direct corollary of the $m$th-order $L^p$ Petty's projection inequality. We list it here again for the convenience of the reader.

\noindent {\bf Theorem~\ref{t:NewLpBPC} (The $m$th-order $L^p$ Busemann-Petty inequality).}
\hfill\break
{\it Fix $n$ and $m$ in $\N$. 
Let $L\subset \M[n,m]$ be a compact domain with positive volume, $Q \!\in\! \conbodo[1,m]$, and $p \geq 1$. Then
    \begin{equation}
     \frac{\vol(\G L)}{\Vol(L)^{1/m}} \geq \frac{\vol(\G \PP \B)}{\Vol(\PP \B)^{1/m}}.
\end{equation}
If $L$ is a star body or a compact domain with piecewise smooth boundary, then there is equality if and only if $L = \PP E$ up to a set of zero volume for some origin symmetric ellipsoid $E \in \conbodio[n,1]$.}

\begin{proof} First, we assume that $L \in \sta$. Applying Theorem~\ref{t:PPIGeneral} to the body $\G L,$ one has the bound
  $$\Vol(\PP \G L ) \vol(\G L )^{\frac{nm}{p}-m} \leq \vol(\B)^{\frac{nm}{p}-m}\Vol(\PP \B ).$$
  Combining this bound with Proposition~\ref{p:BP_P_rel}, one obtains \begin{equation}
    \vol(\B)^{\frac{nm}{p}-m}\Vol(\PP \B ) \geq \left(\frac m {nm+p}\right)^{\frac{nm}{p}} \left(\frac{\vol(\G L )}{\Vol(L)^{1/m}}\right)^{-m}.
  \end{equation} The equality conditions are inherited from Theorem~\ref{t:PPIGeneral} and Proposition \ref{p:BP_P_rel}.
  From the fact that Lemma~\ref{l:class} shows
  \begin{equation}
    \vol(\G \PP \B )=\left(\frac m {\vol(\B)(nm+p)} \right)^{\frac np }\vol(\B),
  \end{equation}
  the claim follows.

Assume now that $L \subset \M[n,m]$ is a compact domain with positive volume. For each fixed $y \in \S$, consider the set $L_{y} = \{\alpha \in [0,\infty) \colon \alpha y \in L\}$. Then, define the function $\ell(\alpha) = \frac{\alpha^{nm}}{nm}$ and let $M$ be the $n\times m$ dimensional star-shaped set with radial function 
\[
\rho_{M}(y) = \ell^{-1}(\vol[1](\ell(L_{y}))).
\]
Moreover, observe that $\Vol(M) = \Vol(L)$. Also, for any $v \in \M[n,1]$, 
\begin{align*}
\int_{L} h_Q(v^t.x)^pdx &= \int_{\S}\int_{L_{y}} h_Q(v^t.y)^p \alpha^p \alpha^{nm-1}d\alpha dy\\
&=\int_{\S}h_Q(v^t.y)^p  \int_{\ell(L_{y})} (nm \beta)^{\frac{p}{nm}}d\beta dy\\
&=(nm)^{\frac{p}{nm}} \int_{\S} h_Q(v^t.y)^p \int_{\ell(L_{y})} \beta^{\frac{p}{nm}}d\beta dy,
\end{align*}
where, in the second line, we applied the change of variables $\beta = \ell(\alpha)$ and $d\beta = \alpha^{nm-1} d\alpha$. 

On the other hand, for each $v \in \M[n,1]$, we additionally observe that 
\begin{align*}
\int_M h_Q(v^t.x)^p dx &= \int_{\S}\int_0^{\rho_M(y)} h_Q(v^t.y)^p \alpha^p \alpha^{nm-1} d\alpha dy\\
&= \int_{\S}h_Q(v^t.y)^p \int_0^{\ell(\rho_M(y))} (nm \beta)^{\frac{p}{nm}}d\beta dy\\
&= (nm)^{\frac{p}{nm}} \int_{\S}h_Q(v^t.y)^p \int_0^{\vol[1](\ell(L_{y})))}\beta^{\frac{p}{nm}}d\beta dy.
\end{align*}
Applying \cite[Theorem~1.14]{LiebAnalysis}, we see that 
\[
\int_{\ell(L_{y})} \beta^{\frac{p}{nm}} d\beta \geq \int_0^{\vol[1](\ell(L_{y}))} \beta^{\frac{p}{nm}}d\beta
\]
for each fixed $y \in \S$.

Therefore, we have shown that
\begin{equation}\label{eq:extension}
\int_L h_Q(v^t.x)^p dx \geq \int_M h_Q(v^t.x)^p dx \quad \text{for all } v \in \M[n,1],
\end{equation}
which implies that $\G L \supseteq \G M$. Since the proof of the inequality in the first part of the proof still holds for compact, star-shaped sets with integrable radial functions, we conclude 
\[
\frac{\vol(\G L)}{\Vol(L)^\frac 1m} \geq \frac{\vol(\G M)}{\Vol(M)^\frac 1m} \geq \frac{\vol(\G \PP \B )}{\Vol(\PP \B )^{1/m}}. 
\]

If $L$ is a compact domain and there is equality for $L$ in the above inequality, then there is equality in \eqref{eq:extension}, which means that $[0,\vol[1](\ell(L_{y}))] = \ell(L_{y})$ for almost every $y \in \S$. Suppose $L$ has a piecewise smooth boundary. Then, we deduce that $M$ is a star body and $M\!=\!L$ up to a set of zero volume. The equality classification follows from the first part of the proof. 
\end{proof}

\subsection{Extremal Cases and Santal\'o inequalities}
\label{sec:Sant}
In this section, we examine the behavior of the operators $\P$, $\PP$ and $\G$ as $p\to \infty$. We introduce the notation
$$\Pi_{Q,\infty} K:=\lim_{p\to\infty}\P K \quad \text{and} \quad \Pi^{\circ}_{Q,\infty} K:=\lim_{p\to\infty}\PP K,$$
where the limit is with respect to the Hausdorff metric. By sending $p\to \infty$ and using properties of $p$th averages, we establish the following formula for the support function of $\Pi_{Q,\infty} K$.
\begin{proposition}
\label{p:infsup}
    Let $m,n\in\N$. Then, for any $K\in\conbodio[n,1]$ and $Q\in\conbodo[1,m]$
    \begin{equation}
\label{eq:identity:operatornorms}
h_{\Pinf[Q]K}(x)=\max_{v \in \s} h_Q\left(\frac{v^t}{h_K(v)}.x\right) = \max_{v \in K^\circ} h_Q\left(v^t.x\right).
\end{equation}
\end{proposition}
\begin{proof}
    Begin by observing that
\[h_{\P K }(x) =[n\vol(K)]^\frac 1p \left(\frac 1 {n\vol(K)}\int_{\s} \left(\frac{h_Q(v^t.x)}{h_K(v)}\right)^p h_K(v) d\sigma_{K}(v) \right)^\frac 1p \]
converges pointwise to  $h_{\Pinf[Q]K}(x)=\max_{v \in \text{supp}(\sigma_K)} h_Q\left(\frac{v^t}{h_K(v)}.x\right)$ as $p\to\infty$. One can show using elementary properties of convex sets that the maximum over $\text{supp}(\sigma_K)$ coincides with the maximum over $\s$. However, we present a simpler argument via approximation. We begin by approximating $K$ with a sequence of convex bodies $\{K_j\}\subset\conbodio[n,1]$ that have $C^2$ smooth boundary with positive Gauss curvature. In this case, $\text{supp}(\sigma_{K_j})=\s$ for all $j$. Next, observe that, for every $j$, 
\begin{align*}h_{\P K_j }(x) =&[n\vol(K_j)]^\frac 1p 
\\
&\times\left(\frac 1 {n\vol(K_j)}\int_{\s} \left(\frac{h_Q(v^t.x)}{h_{K_j}(v)}\right)^p h_{K_j}(v) d\sigma_{K_j}(v) \right)^\frac 1p \end{align*}
converges pointwise to $\max_{v \in \s} h_Q\left(\frac{v^t}{h_{K_j}(v)}.x\right)$ as $p\to\infty$.
Since we may find $r_1,r_2>0$ such that $r_1 B_2^{n m} \subseteq \P K \subseteq r_2 B_2^{n m}$ for every $p\geq 1$, we may apply dominated convergence and use the continuity of the functions $t\mapsto t^p$ and $t\mapsto t^{1/p}$ for $p\geq 1$ to send $j\to\infty$ and the claim follows.

\end{proof}

The following isoperimetric inequality for $\Pi^\circ_{Q,\infty} K$ is immediate from Theorem~\ref{t:PPIGeneral} and Proposition~\ref{p:infsup} by sending $p\to\infty$. 

\begin{corollary}
\label{cor:PPInfty}
Let $n,m \in \N$. Then, for any $K \in \conbodio[n,1]$ and $Q \in \conbodo[1,m]$,
\[\Vol(\Pi^{\circ}_{Q,\infty} K) \vol(K)^{-m} \leq \Vol(\Pi_{Q,\infty}^\circ \B) \vol(\B)^{-m},\]
with equality if $K$ is an origin symmetric ellipsoid.
\end{corollary}

  We next show that Corollary~\ref{cor:PPInfty} implies Theorem~\ref{t:operatornorms}. Recall that $B_{E,F}$ is the unit ball of the norm $\|x\|_{E,F} = \max_{x \in E} \|x.v\|_F$. 

\vspace{1mm}
\noindent {\bf Theorem~\ref{t:operatornorms}.}$\,$ {\it Let $n,m \in \N$. Then, for any pair of convex bodies $E \in \conbodio[n, 1]$ and $F \in \conbodio[1, m]$, 
\[\Vol(B_{E,F}) \vol[m](F)^{-n} \leq \Vol(B_{E, \B[m]}) \vol[m](\B[m])^{-n},\]
with equality if $F$ is an origin symmetric ellipsoid. Moreover, if $E$ has center of mass at the origin, then
\[\Vol(B_{E,F}) \vol(E)^{m} \leq \Vol(B_{\B, F}) \vol(\B)^{m},\] 
with equality if and only if $E$ is an origin symmetric ellipsoid.
}
\begin{proof}
We will interchange the roles of $K\in \conbodio[n,1]$ and $Q\in \conbodio[1,m]$ in Corollary~\ref{cor:PPInfty}. Take $K \subseteq \M[m,1]$ to be the set $F \subseteq \R^m$, identified as a set of column vectors, and take $Q \subseteq \M[1,n]$ for which 
$Q^t \subseteq \M[n,1]$ is the set $E \subseteq \R^n$, identified as a set of column vectors.
Recall the shape of the sets, $B_{E,F} \subseteq \M[m,n]$ and $\Pinf[Q]K \subseteq \M[m,n]$. Let $x \in \M[m,n]$. By \eqref{eq:identity:operatornorms},
\begin{align}
    h_{\Pinf[Q]K}(x)
    &= \max_{v\in F^\circ} h_{E^t}(v^t.x) = \max_{v\in F^\circ} \max_{l\in E^t} v^t.x.l^t \\
    &= \max_{l\in E^t} \max_{v\in F^\circ} v^t.x.l^t = \max_{l\in E} \|x.l\|_F \\
    &= \|x\|_{E,F},
\end{align}
and thus the first part of the theorem follows directly from Proposition \ref{cor:PPInfty}. Clearly, equality holds if $F$ is an origin symmetric ellipsoid.

For the second part, take $K \subseteq \M[n,1]$ to be the set $E^\circ \subseteq \R^n$, identified as a set of column vectors, and take $Q \subseteq \M[1,m]$ for which 
$Q^t \subseteq \M[m,1]$ is the set $F^\circ \subseteq \R^m$, identified as a set of column vectors.
Again, the shape of the sets are $B_{E,F} \subseteq \M[m,n]$ and $\Pinf[Q]K \subseteq \M$. Let $x \in \M[m,n]$, then we have $K^\circ=(E^\circ)^\circ=E$ and
\begin{align}
    h_{\Pinf[Q]K}(x^t)
    &= \max_{v\in E} h_{(F^\circ)^t}(v^t.x^t) = \max_{v\in E} h_{F}(x.v) = \|x\|_{E,F}.
\end{align}
By Corollary~\ref{cor:PPInfty}, and since taking transpose preserves volume in $\M$, we deduce
\[\Vol(B_{E,F}) \vol(E^\circ)^{-m} \leq \vol(B_{\B,F}) \vol(\B)^{-m}.\]
Since $E$ has center of mass at the origin, we may apply the Santal\'o inequality (see \eqref{eq:BS} below) to obtain
\[\vol(E^\circ)^{-m} \geq \vol(E)^m \vol(\B)^{-2m},\]
and the second part of the theorem follows; the characterization of equality also follows from \eqref{eq:BS} below.
\end{proof}

We now return to the Busemann-Petty centroid inequality and analyze $\G L $ as $p\to\infty$. We recall that
$$h_{\G L }(\xi)= \left(\frac 1 {\Vol(L)}\int_L h_Q(\xi^t.x)^p dx \right)^{\frac 1p }.$$
We introduce the following notation: 
$$\Gamma_{Q,\infty} L=\lim_{p\to \infty}\Gamma_{Q,p} L,$$
where the limit is with respect to the Hausdorff metric. By sending $p\to\infty,$ we have, for a compact domain $L\subset\M$ with positive volume, that $\Gamma_{Q,\infty} L$ is a convex body in $\M[n,1]$ defined via the support function
\begin{equation}
\label{eq:higher_order_L}
h_{\Gamma_{Q,\infty} L}(\xi)=\max_{x\in L}h_Q(\xi^t.x). \end{equation}
With this definition in mind, we immediately obtain the following from Theorem~\ref{t:NewLpBPC}.
\begin{corollary}[Busemann-Petty inequality, $p=\infty$]
\hfill\break
\label{cor:LIBPC}
	Let $L \subset \M$ be a compact domain with positive volume and $Q \in \conbodo[1,m]$. Then
\begin{equation}
 \vol(\Gamma_{Q,\infty} L) \geq \left(\frac{\Vol(L)}{\Vol(\Pi_{Q,\infty}^\circ \B )}\right)^{1/m}\vol(\B).  
\end{equation}
\end{corollary}

We recall that the Santal\'o point of $K\in\conbod[d]$ is given by $$s(K)=\text{argmin}_{z\in\text{int}(K)} \vol[d]((K-z)^\circ).$$ A convex body $K$ is said to be in Santal\'o position if $s(K)=o$. Also, we denote the center of mass of $K$ as $c(K):=\vol[d](K)^{-1}\int_K x dx$, and $K$ is said to be centered if $c(K)=o$. The Santal\'o inequality can then be stated as \cite{MP90,CMP85,SR81,SLA49,Sh3} for $K\in\conbod[d]$,
\begin{equation}\vol[d](K)\vol[d]((K-s(K))^\circ)\leq\vol[d](K)\vol[d]((K-c(K))^\circ) \leq \vol[d](B_2^d)^2,
\label{eq:BS}
\end{equation}
with equality throughout if and only if $K$ is an origin symmetric ellipsoid. In fact, in \eqref{eq:BS}, $s(K)$ and $c(K)$ can be replaced by any $z$ in the so-called \textit{Santal\'o region} of $K$, as studied by Meyer and Werner \cite{MW98}. The statements ``in Santal\'o position" and ``centered" below can also be replaced by the more general assumption that the Santal\'o region contains the origin.

Fix $p\geq 1$ and $Q\in\conbodo[1,m]$. For a compact domain $L\subset\M$ with positive volume, a special case of \eqref{eq:BS} is
\begin{equation}
\begin{split}
\vol(\G L )&\vol((\G L-s(\G L))^\circ)
\\
&\leq \vol(\G L )\vol((\G L-c(\G L))^\circ) 
\\
&\leq \vol(\B)^2,
\end{split}
\label{eq:BS_1}
\end{equation}
with equality throughout if and only if $\G L $ is an origin symmetric ellipsoid $E\in\conbodio[n,1]$. We remark here that one can readily verify (say by using \cite[10.22]{Sh1}) that, if $Q\in\conbodio[1,m]$ is symmetric, then $s(\G L)=o$.


Combining \eqref{eq:BS_1} with Theorem~\ref{t:NewLpBPC} yields the following; this extends results by Lutwak and Zhang \cite{LZ97} and Haberl and Schuster \cite{HS09}. It will be convenient to write $\G^\circ L=(\G L)^\circ$.

\begin{lemma}[The $m$th-order $L^p$ Santal\'o inequality]
\hfill\break
\label{l:hipBS}
  Fix $p\geq 1$ and $Q\in\conbodo[1,m].$ Consider a compact domain $L\subset\M$ with positive volume. Then, if $\G L$ is in Santal\'o position or is centered,
  \begin{equation}
  \begin{split}\Vol(L)^{\frac 1 m }\vol(\G^\circ L)\leq \vol(\B)^2 \frac{\Vol(\PP \B )^{\frac 1 m }}{\vol(\G \PP \B )},
  \end{split}
\label{eq:BS_2}
\end{equation}
with equality throughout if and only if $L = \PP E $ for some origin symmetric ellipsoid $E \in \conbodo$.
\end{lemma}

For a compact domain $L\subset \M$ with positive volume, we define its polar with respect to $Q\in \conbodo[1,m]$ as
$$\Gamma^\circ_{Q,\infty} L:=\lim_{p\to\infty}(\G L )^\circ=(\Gamma_{Q,\infty} L)^\circ,$$
which converges with respect to the Hausdorff metric. We remark by \eqref{eq:higher_order_L} that, if $\xi\in \Gamma^\circ_{Q,\infty} L$, then $\max_{x\in L}h_Q(\xi^t.x) \leq 1$. We can view the map $L\mapsto \Gamma^\circ_{Q,\infty} L$ as a type of duality; it follows from the definition that it is order-reversing. Restricting ourselves to $Q=[-1,1]$ and $L\in\conbodio[n,1]$, the case investigated by Lutwak, Yang and Zhang \cite{LZ97}, one obtains $\Gamma_{[-1,1],\infty} L=\text{conv}(L\cup (-L))$ and $\Gamma^\circ_{[-1,1],\infty} L=L^\circ\cap (-L^\circ)$. Thus, if $L\in\conbodio[n,1]$ is symmetric, one obtains $\Gamma_{[-1,1],\infty} L=L$ and $\Gamma^\circ_{[-1,1],\infty} L=L^\circ$. Next, consider when $Q=[0,1]$ and $L\in\conbodio[n,1]$. Then, for every $\xi\in\s$ one has
$$h_{\Gamma_{[0,1],\infty} L}(\xi)=\max_{x\in L}\langle \xi,x \rangle_+=h_L(\xi).$$
Consequently, for $L\in\conbodio[n,1]$, $\Gamma_{[0,1],\infty} L = L$. This example shows that the classical Santal\'o inequality can be considered a special case of Theorem~\ref{t:hiBS} below, which follows immediately from Lemma~\ref{l:hipBS} by taking limits. We note that Lemma~\ref{l:class} shows, as $p\to\infty$, $$\G \PP \B \to\B$$ in the Hausdorff metric for every $Q\in\conbodo[1,m]$.

\begin{theorem}[The $m$th-order Santal\'o inequality]
\hfill\break
\label{t:hiBS}
  Fix $Q\in\conbodo[1,m].$ Consider a compact domain $L\subset\M$ with positive volume. Then, if $\Gamma_{Q,\infty} L$ is in Santal\'o position or is centered,
  \begin{equation}
  \Vol(L)^{\frac 1 m }\vol(\Gamma^\circ_{Q,\infty} L) \leq \vol(\B) \Vol(\Pi_{Q,\infty}^\circ \B)^{\frac 1 m }.
\label{eq:BS_3}
\end{equation}
\end{theorem}
For an example of Theorem~\ref{t:hiBS} that uses the higher-order structure of the results, suppose $Q=[0,1]^m$ and $L\in\conbodio[n,m]$. Then, by writing $\R^n_i$ for the $i$th copy of $\M[n,1]\cong \R^n$ in the decomposition of $\M[n,m]\cong \R^{nm}$ into $m$ independent products of $\R^n$, one obtains $$\Gamma_{[0,1]^m, \infty} L = \text{conv}\left(\bigcup_{i=1}^m P_{\R^n_i} L\right) \longleftrightarrow \Gamma^\circ_{[0,1]^m, \infty} L = \bigcap_{i=1}^m \left(P_{\R_i^n} L\right)^\circ.$$
Thus, if $L\in\conbodio[n,m]$ is so that $\Gamma_{[0,1]^m, \infty} L$ is in Santal\'o position or is centered, we obtain from \eqref{eq:BS_3}
$$\Vol(L)^{\frac 1 m }\vol\left(\bigcap_{i=1}^m \left(P_{\R_i^n} L\right)^\circ\right) \leq \vol(\B) \Vol(\Pi_{Q,\infty}^\circ \B)^{\frac 1 m }.$$

\section{Applications to Function Spaces}
\subsection{Sobolev Spaces and Asymmetric LYZ Bodies}
\label{sec:functionspaces}

Throughout this section, let $f \colon \M[n,1] \to \R$ be a measurable function, and $p \geq 1$. Denote by $\|f\|_p$, the $L^p$ norm of $f$, i.e.,
$$\|f\|_p=\left(\int_{\M[n,1]}|f(v)|^pdv\right)^{1/p}.$$ Let $L^p(\M[n,1])$ be the set of such functions $f$ with $\|f\|_p < \infty$. The function $f$ is said to be non-constant if, for every $\alpha \in \R$, $\vol[n](\{v \in \M[n,1] \colon f(x) \neq \alpha \}) >0$. If $f$ is differentiable, then we denote by $\nabla f$ the gradient of $f$.  

A function $f \colon \M[n,1] \to \R$ is said to be an $L^p$-Sobolev function if $\nabla f$ exists in the weak sense, that is, there exists a measurable vector map $\nabla f:\M[n,1]\to\M[n,1]$ such that $|\nabla f| \in L^p(\M[n,1])$ and
\begin{equation}\int_{\M[n,1]} f(v) \text{ div} (\psi(v)) dv = - \int_{\M[n,1]} (\nabla f(v))^t. \psi(v) dv
\label{eq_LYZBorel}
 \end{equation}
for every compactly supported, smooth vector field $\psi:\M[n,1] \to \M[n,1]$ (see \cite{Evans}). By $W^{1,p}(\M[n,1])$ we mean the set of all $L^p$ Sobolev functions. 

Given a function $f \in W^{1,p}(\M[n,1])$ which is non-constant, Lutwak, Yang and Zhang showed \cite{LYZ06} via the Riesz representation theorem that for every $p\geq 1$, there exists a unique, finite Borel measure $\mu_{f,p}$ on $\s$, not concentrated on any closed hemisphere of $\s$, satisfying the following change of variables formula 
\begin{equation}\int_{\s}g(u)^pd\mu_{f,p}(u)=\int_{\M[n,1]}g(-\nabla f(v))^pdv
\label{eq:LYZw}
\end{equation}
for every $1$-homogeneous, non-negative function $g$ on $\M[n,1]\setminus\{o\}$. Notice for $p=1$ that the center of mass of $\mu_{f,1}$ is the origin. Indeed, with $\{e_{i,1}\}_{i=1}^n$ the canonical basis of $\M[n,1]$,
\begin{align*}
\int_{\s} u\, d\mu_{f,1}(u)&=\sum_{i=1}^n\left(\int_{\s} u^t.e_{i,1}\, d\mu_{f,1}(u)\right)e_{i,1}
\\
&=\sum_{i=1}^n\left(\int_{\M[n,1]} (\nabla f(v))^t.e_{i,1}\, dv\right)e_{i,1}.
\end{align*}
Using \eqref{eq_LYZBorel} then yields
\begin{align*}
\int_{\s} u\, d\mu_{f,1}(u) &=-\sum_{i=1}^n\left(\int_{\M[n,1]}  f(v) \; \mathrm{div} (e_{i,1}) dv\right)e_{i,1} 
\\
&=\sum_{i=1}^n 0 e_{i,1}= o_{n,1}.
\end{align*}
Thus, we apply Minkowski's existence theorem to $\mu_{f,1}$ and obtain a unique convex body with center of mass at the origin, denoted $\langle f\rangle_1$, which we call the \textit{first asymmetric LYZ body of $f$}.

When $p>1, p\neq n$, we appeal to the resolved $L^p$ Minkowski problem for \cite[Theorem 1.4]{HLYZ05} to obtain a unique \textit{$p$th asymmetric LYZ body of $f$}, denoted $\langle f \rangle_p\in \conbodo[n,1]$ such that 
\begin{equation}
\label{eq:Lp_min_prob}
h_{\langle f \rangle_p}(u)^{p-1}d\mu_{f,p}(-u)=d\sigma_{\langle f \rangle_p}(u).\end{equation} 
We note that, from the statement of \cite[Theorem 1.4]{HLYZ05}, $\langle f\rangle_p\in\conbodio[n,1]$ for $p>n$, and thus $d\mu_{f,p}(-u)=d\sigma_{\langle f \rangle_p,p}(u).$ However, the origin may be on the boundary of $\langle f \rangle_p$ in the range $1<p<n$, and so $\sigma_{\langle f \rangle_p,p}$ is not guaranteed to be well-defined on $\{u\in\s:h_{\langle f \rangle_p}(u)=0\}$. The equating of $\mu_{f,p}$ and $\sigma_{\langle f \rangle_p,p}$ on this set is thus \textit{interdit}. Despite this, we relate $\langle f \rangle_p$ and $\nabla f$ in the following lemma.

\begin{lemma}
\label{l:approx_LYZ}
    Fix $1<p$ and let $f\in W^{1,p}(\Rn)$. Then, the unique body $\langle f\rangle_p \in \conbodo[n,1]$ defined via \eqref{eq:Lp_min_prob} satisfies
   \begin{equation}\vol(\langle f \rangle_p)=\frac{1}{n}\int_{\M[n,1]}h_{\langle f \rangle_p}(\nabla f(v))^pdv.
\label{eq:LYZ}
\end{equation}
\end{lemma}
\begin{proof}
In \eqref{eq:LYZw}, replace the generic function $g(u)$ with $g(-u)$. Then, split the integral on the left-hand side over the zero set of $h_{\langle f \rangle_p}$ to obtain
\begin{equation}
\label{eq:LYZ_2}
\begin{split}
\int_{\M[n,1]}g(\nabla f(v))^p dv &= \int_{\{u\in\s:h_{\langle f \rangle_p}(u)=0\}}g(u)^pd\mu_{f,p}(-u)
\\
&\quad+\int_{\{u\in\s:h_{\langle f \rangle_p}(u)\neq 0 \}}g(u)^pd\mu_{f,p}(-u).
\end{split}
\end{equation}
We now apply \eqref{eq:Lp_min_prob} to the third integral in \eqref{eq:LYZ_2} to obtain
\begin{equation}
\label{eq:LYZ_3}
\begin{split}
\int_{\M[n,1]}g(\nabla f(v))^p dv &= \int_{\{u\in\s:h_{\langle f \rangle_p}(u)=0\}}g(u)^pd\mu_{f,p}(-u)
\\
&+\int_{\{u\in\s:h_{\langle f \rangle_p}(u)\neq 0\}}g(u)^ph_{\langle f \rangle_p}(u)^{1-p}d\sigma_{\langle f \rangle_p}(u).
\end{split}
\end{equation}
Next, replace the generic function $g$ with $h_{\langle f \rangle_p}$ in \eqref{eq:LYZ_3} to obtain
\begin{equation}
\label{eq:LYZ_4}
\int_{\M[n,1]}h_{\langle f \rangle_p}(\nabla f(v))^p dv =\int_{\s}h_{\langle f \rangle_p}(u)d\sigma_{\langle f\rangle_p}(u).
\end{equation}
Finally, conclude with \eqref{eq:mixed_0}.
\end{proof}

From the resolved even $L^p$ Minkowski problem for $p \geq 1, p\neq n$ \cite[Theorem 3.3]{LE93}, there exists a unique symmetric convex body, the \textit{$p$th LYZ body of $f$}, which we will denote as $\langle f \rangle_p^e$, such that $$\frac{d\mu_{f,p}(u) + d\mu_{f,p}(-u)}{2} = d\sigma_{\langle f \rangle_p^e,p}(u).$$ By definition, these bodies satisfy \eqref{eq:LYZw} when $g$ is even and $\mu_{f,p}$ is replaced by $\sigma_{\langle f \rangle_p^e,p}$. The bodies $\langle f \rangle_p^e$ were first studied in \cite{LYZ06} and given their epithet by Ludwig \cite{LUD}. We note that, if $\langle f \rangle_p\in\conbodio[n,1]$, then $\langle f \rangle_p^e$ is its origin symmetric $L^p$ Blaschke body (see \cite[pg 148]{LE93}).

We next define $m$th order projection bodies $\LYZ f$ of a function $f\in W^{1,p}(\Rn)$.

\begin{definition}\label{d:generalLYZbody_2} For $1 \leq p,p\neq n,m\in\N$ and $Q\in\conbodo[1,m]$, we define the $(L^p,Q)$-\textit{projection body of non-constant $f\in W^{1,p}(\M[n,1])$}, $\LYZ f$, via
\begin{equation}
\label{eq:LYZ_body_2}
h_{\LYZ f}(\theta)^p = \int_{\M[n,1]}h_Q((\nabla f(v))^t.\theta )^p dv.
\end{equation}\end{definition}

Notice if $Q$ is symmetric, then $\LYZ f =\P \langle f \rangle_p^e$. For a generic $Q\in \conbodo[1,m]$, if $p=1$ or $p>n$, then $\LYZ f = \P \langle f \rangle_p$. However, for $1<p<n$, this is not necessarily the case due to the aforementioned possibility of $o\in \partial \langle f \rangle_p$. They are still related via \eqref{eq:LYZ_3}:
\begin{equation}
\label{eq:LYZ_body_3}
\begin{split}
\int_{\M[n,1]}h_Q((\nabla f(v))^t.\theta )^p dv &= \int_{\{u\in\s:h_{\langle f \rangle_p}(u)=0\}}h_Q(u^t.\theta )^pd\mu_{f,p}(-u)
\\
&+\int_{\{u\in\s:h_{\langle f \rangle_p}(u)\neq 0\}}h_Q(u^t.\theta )^pd\sigma_{\langle f \rangle_p,p}(u).
\end{split}
\end{equation}
Additionally, they are related via Petty's projection inequality. We of course set $\LYZP f := (\LYZ f)^\circ$.
\begin{proposition}
    For $1 \leq p,p\neq n, f\in W^{1,p}(\M[n,1]), m\in\N$ and $Q\in\conbodo[1,m]$, we have
    \begin{equation}
\label{eq:LYZ_petty_approx_final}
 \Vol(\LYZP f )  \vol(\langle f \rangle_p)^{\frac{nm} p - m}  \leq \Vol(\PP \B ) \vol(\B)^{\frac{nm} p - m}. 
\end{equation}
\end{proposition}
\begin{proof}
For $p=1$ and $p>n$, the inequality follows by applying \eqref{e:PPIGeneral} to $\langle f \rangle_p$. For $1<p<n$, we do the following. 
We begin by recalling the following result from Chou and Wang \cite{CW06}: if $\nu$ is a discrete measure on $\s$ not concentrated in any closed hemisphere, then for $1<p,p\neq n$, there exists a convex polytope $P(\nu,p)\in\conbodio[n,1]$ such that $\nu = \sigma_{P(\nu,p),p}$. Therefore, following \cite[pg. 392-393]{Sh1}, we approximate $\mu_{f,p}$ with a sequence of discrete measures $\nu_k$ that are not concentrated in any closed hemisphere. Thus, we have for each $k$ that $d\nu_k(-u)=d\sigma_{P_k(f,p),p}(u)$ for some $P_k(f,p)\in \conbodio[n,1]$.

    From the definition of weak convergence, we then have that for every $v\in \s$
    $$\int_{\s}(u^t.v)_+d\nu_k(-u) \to \int_{\s}(u^t.v)_+d\mu_{f,p}(-u) >0.$$
    But, this is a sequence of convex functions (in $v$) converging pointwise to a convex function. Thus, the convergence is uniform \cite[Theorem 10.8]{RTR70}. We deduce that there exists a $c>0$ independent of $k$ so that
    $$\int_{\s}(u^t.v)d\nu_k(-u) > c >0.$$
    It follows from \cite[Lemmas 2.2 and 2.3]{LYZ04} that each $P_k(f,p)$ has positive volume and is contained in $RB_2^n$, where $R>0$ is independent of $k$. Therefore, from Blaschke's selection theorem \cite[Theorem 1.8.6]{Sh1}, there exists a subsequence, which we still denote as $P_k(f,p)$ such that $P_k(f,p)\to \langle f \rangle_{p}$ in the Hausdorff metric. Then, from Theorem~\ref{t:PPIGeneral} applied to $P_k(f,p)$, we obtain
\begin{equation}
\label{eq:LYZ_petty_approx}
\Vol(\PP \!P_k(\!f,p)\! ) \vol(\!P_k(f,p)\!)^{\frac{nm} p \!-m} 
\!\!\leq\! \Vol(\PP\! \B ) \vol(\!\B\!)^{\frac{nm} p \!- m}\!.
\end{equation}
From the polar formula for volume \eqref{eq:polar} and the relation between $\nu_k$ and $P_k(f,p)$, we have
\begin{equation}
\label{eq:LYZ_polytopes_volume}
    \Vol(\PP P_k(f,p) ) = \frac{1}{nm}\int_{\S}\left( \int_{\s} h_Q(u^t.\theta)^p d\nu_k(-u)\right)^{-\frac{nm}{p}}d\theta.
\end{equation}
Notice that, for a fixed $k$, the function
$$\theta\mapsto h_{\P P_k(f,p)}(\theta) = \int_{\s} h_Q(u^t.\theta)^p d\nu_k(-u)$$
 is strictly positive and Lipschitz continuous on $\S$, and therefore obtains a strictly positive minimum. This is true also for $\theta\mapsto h_{\LYZ f}(\theta)$. Consequently, since this is a sequence of strictly positive convex functions converging uniformly to a strictly positive convex function, we may take the limit in $k$ and apply the dominated convergence theorem to \eqref{eq:LYZ_polytopes_volume} to obtain
\begin{align*}
    \lim_{k\to\infty}\Vol(\PP P_k(f,p) ) &= \frac{1}{nm}\int_{\S}\left( \int_{\s} h_Q(u^t.\theta)^p d\mu_{f,p}(-u)\right)^{-\frac{nm}{p}}d\theta
    \\
    &=\frac{1}{nm}\int_{\S}\left( \int_{\M[n,1]} h_Q((\nabla f(v))^t.\theta)^p dv\right)^{-\frac{nm}{p}}d\theta
    \\
    &=\Vol(\LYZP f).
\end{align*}
Thus, the desired inequality follows from \eqref{eq:LYZ_petty_approx} by taking the limit in $k$.

\end{proof}

 \subsection{The $m$th-order $L^p$ affine Sobolev inequalities} \label{sec:functionalineqs}

In this section, we prove Theorem~\ref{t:GeneralAffineSobolev}. Before proving the former, we require the following lemma. For $f\in W^{1,p}(\M[n,1]),$ its $L^p$ Sobolev norm is given by
$$\|f\|_{W^{1,p}}^p = \|f\|^p_p + \|\nabla f\|^p_p.$$
One then denotes $W_0^{1,p}(\M[n,1])$ as the closure of the vector space of compactly supported functions with respect to the $L^p$ Sobolev norm.

\begin{lemma}[Anisotropic Sobolev inequality]
\label{l:asob}
Let $K \in \conbodo[n,1]$, fix $p\in [1,n)$ and let $p^* = \frac{pn}{n-p}$. Consider a function $f\in W_0 ^{1,p}(\M[n,1])$ non-constant. Then 
$$a_{p,n}\omega_n^{-\frac{1}{n}}\|f\|_{p^\star}\vol(K)^\frac 1n \leq  \left(\int_{\M[n,1]}h_K (\nabla f(z))^pdz\right)^\frac 1p .$$
Let $K$ be an origin symmetric convex body. If $p>1$, define the function
$$f_{p,K}(x):=\left(\alpha_p+\|x\|_K^\frac p {p-1}\right)^{-\frac{n-p} p }.$$
Here, $\alpha_p$ is so that $\|f_{p,K}\|_{p^\star}=1$. Then, for $p>1$, there is equality in the above inequality if and only if there exist $\alpha,\beta\!>\!0$ and $x_0\!\in\!\M[n,1]$ such that \begin{equation}
\label{eq:lsob_eq}
f(x)=\alpha f_{p,\beta K}(x-x_0).\end{equation}
\end{lemma}

In Lemma~\ref{l:asob}, one has $a_{1,n}=n\omega_n^\frac{1}{n}$, and, for $p>1$,
\begin{equation}
    a_{p,n} = n^\frac{1}{p}\left(\frac{n-p}{p-1}\right)^\frac{p-1}{p}\left(\frac{\omega_n}{\Gamma(n)}\Gamma\left(\frac{n}{p}\right)\Gamma\left(n+1-\frac{n}{p}\right)\right)^\frac{1}{n},
    \label{eq:sobolev_cons}
\end{equation}
which, in fact, is the sharp constant from the Aubin-Talenti $L^p$ Sobolev inequality \cite{Aubin1,Talenti1}. When $K = B_2^n$, Lemma~\ref{l:asob} was done in \cite{FF60} by Federer and Fleming  and \cite{VGM60} by  Maz'ja. When $K$ is origin symmetric, Lemma~\ref{l:asob} was done in \cite{MS86} by Gromov, and  extended by  Cordero-Erausquin, Nazaret and Villani in \cite{CENV04} with equality conditions. In \cite{AFTL97}, Alvino, Ferone, Trombetti and Lions showed the inequality for the non-symmetric case of $K\in\conbodio[n]$. The version of Lemma~\ref{l:asob} that we present here for $K\in \conbodo[n,1]$ follows. Indeed, if $\partial K$ contains the origin, let $\{x_i\}$ be a sequence so that $x_i\to o$ and $(K-x_i)\in\conbodio[n,1]$. The lemma then follows from the translation invariance of the left-hand side of the inequality and using the dominated convergence theorem (it is easy to see that the right-hand side is bounded by applying the change of variables \eqref{eq:LYZw}). 

The reader may wonder about equality in the case when $p=1$ in Lemma~\ref{l:asob}. Here, the extremal function is a multiple of $\chi_L$, where $L$ is any convex body homothetic to $K$. However, characteristic functions are not in $W_0^{1,1}(\Rn)$, instead they are in $BV(\Rn)$ (the so-called functions of bounded variation). An extension of Lemma~\ref{l:asob} to the class of functions of bounded variations was done in \cite{HLPRY23} by expanding on \cite{AFTL97, CENV04}. The asymmetric LYZ bodies also exist for $f\in BV(\Rn)$ (see \cite{HLPRY23,KS21,TW12}). 

As an example of Lemma~\ref{l:asob}, we can apply it to $K=\langle f \rangle_p$ and obtain via Lemma~\ref{l:approx_LYZ} that \begin{equation}
\label{eq:bound_LYZ}
\vol(\langle f\rangle_p)^{\frac{1}{p^\star}}\geq \frac{a_{p,n}}{n^\frac{1}{p}\omega_n^{\frac{1}{n}}}\|f\|_{p^\star}.\end{equation}
In fact, combining \eqref{eq:bound_LYZ} with Proposition~\ref{eq:LYZ_petty_approx_final} immediately yields the inequality in Theorem~\ref{t:GeneralAffineSobolev}. To also show sharpness, we will instead use a different approach utilizing Theorem~\ref{t:NewLpBPC}. We note that Theorem~\ref{t:GeneralAffineSobolev} can be easily extended to $f\in BV(\Rn)$ when $p=1$. For ease of presentation, we suppress this fact and focus only on $p>1$.

\noindent {\bf Theorem~\ref{t:GeneralAffineSobolev}.}{\it
  $\,$ Fix $m,n \in \N,$ $p \in [1,n)$, and $Q \in \conbodo[1,m]$. Set $p^* = \frac{pn}{n-p}$. Consider a non-constant function $f\in W_0 ^{1,p}(\M[n,1])$. Then:
\begin{equation}
\label{eq:the_sob_again}
d_{n,p}(Q)\left(\int_{\S} \left( \int_{\M[n,1]} h_Q((\nabla f(z))^t.\theta)^pdz \right)^{-\frac{nm} p } d\theta\right)^{-\frac 1 {nm}} \geq a_{p,n}\|f\|_{p^*},
\end{equation}
where 
\[
d_{n,p}(Q) := (n\omega_n)^{\frac{1}{p}}\left(nm\Vol(\PP \B )\right)^\frac 1 {nm}.
\]
If $p>1$, then there is equality if and only if there exist $\alpha>0$, $A\in GL_n(\R)$ and $v_0\in \M[n,1]$ such that
\[
f(v)=(\alpha+|A.(v-v_0)|^{\frac p {p-1}})^{-\frac{n-p} p }. \]
If $p=1$, then the inequality can be extended to functions of bounded variation, in which case equality holds if and only if $f$ is a multiple of $\chi_E$ for some ellipsoid $E\in\conbod[n,1]$.}
\begin{proof}

For ease, we set
\begin{equation}
\begin{split}
    \mathcal{E}_p(Q,f) :&= d_{n,p}(Q)\left(\!\int_{\S}\! \left(\! \int_{\M[n,1]} h_Q((\nabla f(z))^t.\theta)^pdz \right)^{-\frac{nm} p }\! d\theta\!\right)^{-\frac 1 {nm}}
    \\
    &=d_{n,p}(Q)(nm)^{-\frac{1}{nm}}\Vol(\LYZP f)^{-\frac{1}{nm}},
    \label{eq:relation_E}
    \end{split}
\end{equation}
where the last equality follows from applying polar coordinates \eqref{eq:polar} to the volume of the body $\LYZP f$ defined via Definition~\ref{d:generalLYZbody_2}. We note that the proof of the inequality and the equality characterization for $p=1$ is similar to the proof of the equality conditions in the affine Sobolev inequality in the prequel work \cite{HLPRY23}, so we focus on $p>1$. First, we claim that
\begin{equation}
\begin{split}
\label{eq:yet_another_formula}
&d_{n,p}(Q)^{-1}\mathcal{E}_p(Q,f) 
\\
&= \left(\Vol(\LYZP f)(nm+p) \int_{\M[n,1]}h_{\G \LYZP f}(\nabla f(x))^p dx\right)^{-\frac{1}{nm}}.
\end{split}
\end{equation}
Indeed, we have from Fubini's theorem and polar coordinates
    \begin{align*}
        &\int_{\M[n,1]}h_{\G \LYZP f}(\nabla f(x))^p dx 
        \\
        &= \frac 1 {\Vol(\LYZP f)}\int_{\LYZP f} \int_{\M[n,1]} h_Q((\nabla f(y)^t.x)^p dy dx
        \\
        &=\frac 1 {\Vol(\LYZP f)}\int_{\LYZP f} \|x\|^p_{\LYZP f} dx
        \\
        &=\frac 1 {\Vol(\LYZP f)} \int_{\S} \|\theta\|^p_{\LYZP f} \int_0^{\|\theta\|^{-1}_{\LYZP f}}r^{nm-1+p}drd\theta
        \\
        &=\frac 1 {\Vol(\LYZP f)(nm+p)} \int_{\S} \|\theta\|^{-nm}_{\LYZP f} d\theta
        \\
        &=\frac 1 {\Vol(\LYZP f)(nm+p)} \int_{\S}\left(\int_{\M[n,1]}h_Q((\nabla f(x))^t.\theta )^p dx\right)^{-\frac{nm}{p}}d\theta
        \\
        &=\frac{\mathcal{E}_p(Q,f)^{-nm}d_{n,p}(Q)^{nm}}{\Vol(\LYZP f)(nm+p)}.
    \end{align*}
 Next, we have \begin{equation}
 \label{eq:ready_for_asbo}
 \mathcal{E}^p_p(Q,f) \geq \left(\frac{\omega_n}{\vol(\G \LYZP f)}\right)^\frac{p}{n}\int_{\M[n,1]}h_{\G \LYZP f}(\nabla f(x))^p dx\end{equation}
    with equality if and only if $\G \LYZP f = E$ for some origin symmetric ellipsoid $E\in\conbodio[n]$. Indeed, by using, \eqref{eq:yet_another_formula}, \eqref{eq:relation_E}, Theorem~\ref{t:NewLpBPC}, and \eqref{eq:sharp}:
    \begin{align}
        &\mathcal{E}^p_p(Q,f) = \mathcal{E}_p(Q,f)^{nm+p} \mathcal{E}_p(Q,f)^{-nm}
        \\
        & = \mathcal{E}_p(Q,f)^{nm+p} d_{n,p}(Q)^{-nm} \Vol(\LYZP f)(nm+p)
        \\
        &\quad\quad\quad\times\int_{\M[n,1]}h_{\G \LYZP f}(\nabla f(x))^p dx
        \\
        &=d_{n,p}(Q)^p(nm)^{-\frac{nm+p}{nm}}(nm+p)\vol(\LYZP f)^{-\frac{p}{nm}}
        \\&\quad\quad\quad\times\int_{\M[n,1]}h_{\G \LYZP f}(\nabla f(x))^p dx
        \\
        &\geq \omega_n\frac{nm+p}{m}\left(\frac{\vol(\G \PP \B)}{\vol(\G \LYZP f)}\right)^\frac{p}{n}\int_{\M[n,1]}h_{\G \LYZP f}(\nabla f(x))^p dx,
    \end{align}
and then \eqref{eq:ready_for_asbo} follows from \eqref{eq:elipp_cal}. We now use Lemma~\ref{l:asob} in the case $K=\G \LYZP f$: from \eqref{eq:ready_for_asbo}, we obtain
$$\mathcal{E}_p(Q,f) \geq a_{p,n}\|f\|_{p^\star},$$
with equality if and only if $\G \LYZP f = E$ and $f$ has the form \eqref{eq:lsob_eq} with $K=E$. The claim follows, since $\|x\|_E=|A.x|$ for some $A\in GL_n(\R)$.
\end{proof}

{\bf Acknowledgments :} We would like to thank Yiming Zhao for his comments concerning the construction of the $p$th asymmetric LYZ bodies in the range $1\leq p <n$. We sincerely thank the referees for their excellent comments, which have greatly improve both the presentation and the overall quality of our work. 

{\bf Funding:} The first named author was supported Grant RYC2021-032768-I, funded by the Ministry of Science and Innovation/State Research Agency/10.13039/ 501100011033 and by the E.U. NextGenerationEU/Recovery, Transformation and Resilience Plan. The second named author was supported in part by the U.S. NSF Grant DMS-2000304, by the Chateaubriand Scholarship offered by the French embassy in the United States, and by the Fondation Sciences Math\'ematiques de Paris Postdoctoral program. This material is based upon work supported by the National Science Foundation under Grant No. DMS-1929284 and Simons Foundation under Grant No. 815891 while the fourth named author was in residence at the Institute for Computational and Experimental Research in Mathematics in Providence, RI, during the Discrete Optimization: Mathematics,  Algorithms, and Computation semester program in Spring 2023. The fifth named author was supported in part by NSERC.

\bibliographystyle{acm}
\bibliography{references}
\end{document}